\documentclass[11pt]{amsart}
\usepackage{wrapfig}
\usepackage{graphicx}
\usepackage[dvipsnames]{xcolor}
\usepackage{amsmath}
\usepackage{amssymb}
\usepackage{amsfonts}
\usepackage{latexsym, amsthm, mathrsfs}
\usepackage{hyperref, breakurl}
\usepackage[latin1]{inputenc}
\usepackage[T1]{fontenc} 
\usepackage{lmodern} 
\usepackage{fancybox}
\usepackage{amscd}
\usepackage[shortlabels]{enumitem} 
\input{xy}
\usepackage[all]{xy}
\setlist{itemsep=.1em}

\newcommand{\nc}{\newcommand}
\nc{\nothing}[1]{} 
\nc{\commenth}[1]{\color{blue}#1 \color{black}}

\nc{\commentblue}[1]{\color{blue}#1 \color{black}}

\nc{\marginparr}[1]{\marginpar{#1}} 
\nc{\monday}[1]{\marginpar{\tiny 09/20 H: #1}} 

\nc{\Miller}{\mathbb M}
\nc{\Laver}{\mathbb L}
\nc{\Landver}{\mathbb D}

\numberwithin{equation}{section}
\newtheorem{theorem}{Theorem}[section]
\newtheorem{lemma}[theorem]{Lemma}

\theoremstyle{definition} 

\newtheorem{definition}[theorem]{Definition}

\newtheorem{problem}[theorem]{Problem}
\theoremstyle{plain} 

\newtheorem{proposition}[theorem]{Proposition}
\newtheorem{corollary}[theorem]{Corollary}
\newtheorem{fact}[theorem]{Fact}

\newtheorem*{maintheorem*}{Main Theorem}
\newtheorem*{conjecture*}{Conjecture}
\newtheorem*{theorem*}{Theorem}
\newtheorem*{proposition*}{Proposition}
\newtheorem*{corollary*}{Corollary}

\newtheorem{remark}[theorem]{Remark} 

\newtheorem{observation}[theorem]{Observation}
\theoremstyle{remark}  
\newtheorem*{remarks*}{Remarks}
\newtheorem*{remark*}{Remark}
\newtheorem*{claim*}{Claim}

\nc{\nco}{\DeclareMathOperator}
\nco{\rk}{rk}
\nco{\ppower}{pp}
\nco{\pcf}{pcf} 
\nco{\tcf}{tcf} 
\nco{\tlim}{tlim} 
\nco{\limtext}{lim} 
\nco{\prodt}{{\textstyle \prod}}
\nco{\symdiff}{\triangle}
\nco{\dom}{dom}
\nco{\card}{card}
\nco{\lh}{lh}
\nco{\lt}{lt}
\nco{\lgg}{lg}
\nco{\hgt}{ht}
\nco{\rge}{range}
\nco{\otp}{otp}
\nco{\trunk}{tr}
\nco{\cf}{cf}
\nco{\cof}{cof}
\nco{\nex}{next}
\nco{\reduction}{red}
\nco{\supt}{supt}
\nco{\supp}{supp}
\nco{\Lim}{Lim}
\nco{\Leb}{Leb}
\nco{\id}{id}
\nco{\ro}{ro}
\nco{\Levy}{Coll}

\nco{\OSucc}{OSucc} 
\nco{\Succ}{Succ}
\nco{\Split}{Split}
\nco{\Stem}{Stem}
\nco{\acc}{acc}

\nc{\MA}{\mbox{\sf MA}}
\nc{\PFA}{\mbox{\sf PFA}}
\nc{\OCA}{\mbox{\sf OCA}}
\nc{\GCH}{\mbox{\sf GCH}}
\nc{\CH}{\mbox{\sf CH}}
\nc{\ZFC}{\mbox{\sf ZFC}}
\nc{\SCH}{\mbox{\sf SCH}}
\nc{\ZF}{\mbox{\sf ZF}}

\nc{\On}{{\rm On}}
\nc{\rest}{\restriction}
\nc{\nrest}{\!\rest\!}
\nc{\bF}{\mathbb F}
\nc{\F}{\mathbb F}
\nc{\bH}{\mathbb H}
\nc{\R}{\mathbb R}
\nc{\M}{\mathbb M}
\nc{\bQ}{\mathbb Q}
\nc{\cA}{\mathcal A}
\nc{\cP}{\mathcal P}
\nc{\cC}{\mathcal C}
\nc{\cF}{\mathcal F}
\nc{\bP}{\mathbb P}
\nc{\bT}{\mathbb T}
\nc{\bG}{\mathbf G}
\nc{\concat}{{}^\smallfrown{}}   
\nc{\such}{\: | \:}                   

\nc{\ms}{\medskip}
\nc{\eps}{\varepsilon}
\nc{\seq}[2]{\langle #1 \such #2 \rangle}
\DeclareMathOperator{\Add}{Add}

\begin{document}

\title[Tree Forcings]{Distributivity and Minimality in Perfect Tree Forcings for Singular Cardinals}
\author{Maxwell Levine}
\address{Albert-Ludwigs-Universit\"at Freiburg,
Mathematisches Institut, Abteilung f\"ur math. Logik, Ernst--Zermelo--Stra\ss e~1, 
79104 Freiburg im Breisgau, Germany}
\email{maxwell.levine@mathematik.uni-freiburg.de}
\author{Heike Mildenberger}
\address{Albert-Ludwigs-Universit\"at Freiburg,
Mathematisches Institut, Abteilung f\"ur math. Logik, Ernst--Zermelo--Stra\ss e~1, 
79104 Freiburg im Breisgau, Germany}
\email{heike.mildenberger@math.uni-freiburg.de}


\begin{abstract} Dobrinen, Hathaway and Prikry studied a forcing $\bP_\kappa$ consisting of perfect trees of height $\lambda$ and width $\kappa$ where $\kappa$ is a singular $\omega$-strong limit of cofinality $\lambda$. They showed that if $\kappa$ is singular of countable cofinality, then $\bP_\kappa$ is minimal for $\omega$-sequences assuming that $\kappa$ is a supremum of a sequence of measurable cardinals. We obtain this result without the measurability assumption.

Prikry proved that $\bP_\kappa$ is $(\omega,\nu)$-distributive for all $\nu<\kappa$ given a singular $\omega$-strong limit cardinal $\kappa$ of countable cofinality, and Dobrinen et al.\ asked whether this result generalizes if $\kappa$ has uncountable cofinality. We answer their question in the negative by showing that $\bP_\kappa$ is not $(\lambda,2)$-distributive if $\kappa$ is a $\lambda$-strong limit of uncountable cofinality $\lambda$ and we obtain the same result for a range of similar forcings, including one that Dobrinen et al.\ consider that consists of pre-perfect trees. We also show that $\bP_\kappa$ in particular is not $(\omega,\cdot,\lambda^+)$-distributive under these assumptions.

While developing these ideas, we address natural questions regarding minimality and collapses of cardinals.\end{abstract}

\maketitle

\section{Introduction and Background}\label{intro-sec}

The study of forcing in set theory is bound to the study of complete Boolean algebras through the fact that every forcing poset can be identified with its regular open algebra. Functions added in the forcing extension correspond to failures of distributivity in the regular open algebra, and sub-extensions correspond to complete subalgebras. Questions on distributivity and minimality thus motivated some of the early development of forcing \cite{Scott_boolean}. Solovay asked the question: For which cardinals $\kappa$ is there a Boolean algebra that is $(\omega,\nu)$-distributive for all $\nu < \kappa$ but not $(\omega,\kappa)$-distributive? Namba \cite{Namba_main} answered this question for regular cardinals $\kappa \ge \aleph_2$ using a perfect tree forcing and an assumption on cardinal exponentiation. In unpublished work, Prikry answered Solovay's question for singular cardinals of countable cofinality, also using a perfect tree forcing. Both forcings are minimal for $\omega$-sequences. More recently, Dobrinen, Hathaway, and Prikry \cite{DobrinenHathawayPrikry} developed this line of investigation further and posed an open-ended version of Solovay's question: Given a regular cardinal $\lambda$, for which cardinals $\kappa$ is there a complete Boolean algebra that is $(\lambda,\nu)$-distributive for all $\nu<\kappa$ but not $(\lambda,\kappa)$-distributive? This paper focuses on the case where $\kappa$ is singular and $\omega<\lambda = \cf(\kappa)$ by ruling out many natural candidates.

In pursuing this question, the most natural forcings to consider are higher tree forcings of the sort originally studied by Kanamori using additional requirements on splitting that do not pertain to the countable cases \cite{Kanamori-higher-sacks}. This includes the ``vertical'' requirement that the trees are both closed under ascending sequences of splitting nodes and the ``horizontal'' requirement that they split into sufficiently closed filters. Dobrinen et al.\ point out that the main obstacle to studying higher tree forcings in full generality has to do with the fusion sequences typically used to work with tree forcings---namely, that fusion sequences of length $\cf(\kappa)$ may break down at the first countable step. Our work here considers what happens both with and without Kanamori's splitting requirements. We show that even with the splitting requirements, a variety of perfect tree forcings for singular cardinals $\kappa$ of uncountable cofinality fail to be $(\cf(\kappa),2)$-distributive. This answers a question of Dobrinen et al$.$ regarding the generalization of Prikry's perfect tree forcing. We also show countable distributivity fails for all versions but the ones consisting of trees that obey both the vertical and horizontal requirements. In other words, our findings suggest that failures of countable distributivity predominate in the higher cases when one does not deliberately ensure the viability of fusion sequences.

Most of the introduction to this paper will be a review of definitions of the concepts we are working with. We consider singulars of countable cofinality in Section~\ref{countable-sec}, in which we present a short proof of Prikry's theorem on $(\omega,\nu)$-distributivity and obtain $\omega$-minimality for $\bP_\kappa$ given an $\omega$-strong limit $\kappa$ of countable cofinality. In Section~\ref{bernstein-sec} we move to the case of uncountable cofinality and present a technique whereby, given a perfect tree $T$ of singular width $\kappa$ and height $\lambda>\omega$, we consider Bernstein sets on the cofinal branches $[T \! \rest \! \delta]$ for some $\delta<\lambda$ to define useful functions $\lambda \to \On$ in the extensions. We use this technique to prove that $\bP_\kappa$ is not $(\cf(\kappa),2)$-distributive. Then we settle some natural questions on the closed versions of perfect tree forcings and obtain a non-minimality result for one of the versions that is not closed. In Section~\ref{balcar-simon-sec} we define a concept of slow fusion that allows us to work with higher tree forcings that are not as tame as those usually considered in the literature. We use this together with an old trick of Balcar and Simon to obtain various failures of countable distributivity. In Section~\ref{3par-sec} we obtain failures of three-parameter distributivity. A curiosity resulting from this section is that the regular open algebra of the poset of perfect binary trees in a regular uncountable cardinal $\kappa$ (a higher perfect tree forcing without Kanamori's restrictions) is isomorphic to that of the L{\'e}vy collapse $\Levy(\omega,\kappa^+)$ if $2^\kappa = \kappa^+$. Generally, Sections~\ref{balcar-simon-sec} and \ref{3par-sec} are relevant to the cases where $\kappa$ is regular despite this paper's title.

\subsection{Distributivity and Boolean Algebras}

Here we lay down our most important definitions pertaining to Boolean algebras. More details can be found in Jech's textbook \cite{Jech3} and in the Handbook of Boolean Algebras \cite{Jech:distrBA}.

\begin{definition} Given a poset $\bP$, the \emph{regular open algebra} $\ro(\bP)$ of $\bP$ consists of subsets $u \subseteq \bP$ such that:

\begin{enumerate}[(i)]
\item $\forall p \in u, q \le p, q \in u$;
\item $\forall p \notin u, \exists q \le p, \forall r \le q, r \notin u$.
\end{enumerate}
\end{definition}

For a Boolean algebra $B$, we denote by $B^+ = B \setminus \{0_B\}$.

\begin{fact} If $\bP$ is a poset, then $\ro(\bP)$ is a complete Boolean algebra and $\bP$ embeds densely into $\ro(\bP)$.\end{fact}

\begin{definition}[{\cite[Def.17]{BalcarSimonHandbook}}]
 Let $B$ be a Boolean algebra. Let $\theta$, $\mu$, $\nu$ be cardinals such that $\theta, \nu \geq \omega$ and $\mu \geq 2$.
 
 \begin{itemize}
 
 \item  A collection $\cP \subseteq \cP(B^+)$ is called a \emph{matrix} if each member of $\cP$ is a maximal disjoint subset of $B^+$.
 
 \item We say that $B$ is \emph{$(\theta,\nu, \mu)$-distributive} if for every matrix $\cP = \{P_\alpha \such  \alpha \in \theta\}$ such that $|P_\alpha| \leq \nu$ there is some maximal disjoint set $Q \subseteq B^+$ such that for each $q \in Q$ and $\alpha \in \theta$, $|\{p \in P_\alpha \such p \cap q \neq 0\}| < \mu$.
  
\item If $B$ is $(\theta,\nu, 2)$-distributive, then we say that $B$ is $(\theta,\nu)$-distributive.

\item Finally, $B$ is called $\theta$-distributive if it is $(\theta, \nu)$-distributive for any $\nu$.
\end{itemize}
\end{definition}

\begin{fact} A forcing poset $\bP$ is $\theta$-distributive if and only if every function $f:\theta \to \On$ in $V[\bP]$ is already contained in $V$. A poset $\bP$ is $(\theta,\nu)$-distributive if and only if every function $f:\theta \to \nu$ in $V[\bP]$ is already contained in $V$.\end{fact}

\begin{definition} Two posets $\bP_1$ and $\bP_2$ are called \emph{forcing equivalent}, denoted $\bP_1 \simeq \bP_2$, if $\ro(\bP_1) \cong \ro(\bP_2)$. We write $\bP_1 \lessdot \bP_2$ if $\ro(\bP_1)$ is isomorphic to a complete subalgebra of $\ro(\bP_2)$.\end{definition}


\subsection{Basic Notions of Tree Forcings}

We recall some notions pertaining to trees. 

\begin{definition}\label{1.1}
Let $\kappa$ be the supremum of an increasing sequence $\vec \kappa = \langle \kappa_\alpha \such \alpha < \delta \rangle$ of regular infinite cardinals where $\delta$ is an ordinal less than or equal to $\kappa$.

\begin{enumerate}
\item
The set $N=N(\vec \kappa)= \bigcup_{\alpha<\delta}\prod_{\beta<\alpha}\kappa_\beta$ consists of all functions $t$ such that $\dom(t) < \delta$ and $(\forall \alpha \in \dom(t) )(t(\alpha) < \kappa_\alpha)$. We typically refer to such functions as \emph{nodes}. We drop the notation for $\vec \kappa$ when the sequence is fixed in context.

\item
A non-empty subset $T \subseteq N$ that is closed under initial segments is called a \emph{tree}. 

\item 
For an ordinal $\alpha$, the set
$T(\alpha)$ is the set of $t \in T$ with $\dom(t) =\alpha$.

\item 
The \emph{height of a tree $T$} is $\hgt(T) = \min\{\alpha \such T(\alpha) = \emptyset\}$.

\item 

We let $[T] = \{f \colon \hgt(T) \to \kappa \such (\forall \alpha < \hgt(T) f \rest \alpha \in T\}$. 
Elements of $[T]$ are called \emph{cofinal branches}.
\item 
For
$t_1$, $t_2 \in N \cup [N]$ we write $t_1 \sqsubseteq t_2$ if $t_2 \nrest \dom(t_1) = t_1$. The tree order is the relation $\sqsubseteq$. If $t = s \cup \{(\dom(s), \beta)\}$, we write $t = s \concat \langle \beta \rangle$.

\item 
A branch of $T$ is an element of $T \cup \{f \in N \cup [T] \such \forall \alpha \in \dom(f), f \nrest \alpha \in T\}$ that does not have a $\sqsupset$-extension.


\item 
$T \nrest \alpha = \bigcup_{\beta<\alpha}T(\beta)$.

\item 
$T \nrest t = \{s \in T \such s \sqsubseteq t  \vee t \sqsubseteq s\}$


\item 
For $t \in T(\alpha)$ we let 

\vspace{-5mm}
\[ \Succ_T(t) = \{c \such c \in T(\alpha+1) \wedge c \sqsupseteq t\}\]
denote the \emph{set of immediate successors of $t$}, and
\[ \OSucc_T(t) = \{\beta \such t \concat \langle \beta \rangle  \in T(\alpha+1)\}\]
denote the \emph{ordinal successor set of $t$}.

\item We call $t \in T$ a \emph{splitting node} if $|\Succ_T(t)|>1$. 

\item If it exists, 
$\Stem(T)$ is the unique splitting node of $T$ that is comparable to all other nodes. 

\end{enumerate}
\end{definition}

\begin{definition}\label{manners} Throughout this paper we will need to consider different manners of branching. Once again we fix $\vec \kappa = \langle \kappa_\alpha \such \alpha < \delta \rangle$.

\begin{enumerate}

\item A non-empty tree $T \subseteq N(\vec \kappa)$ is called \emph{ever-branching} if for each $\alpha < \delta$ and for each $s \in T$ there is a node $t \sqsupseteq s$ such that 

\vspace{-5mm}

\begin{equation}\label{criterion1}
|\Succ_T(t)| \geq \kappa_\alpha.
\end{equation}
Such a node $t$ is called \emph{$\kappa_\alpha$-splitting}. Note that $\dom(t) > \alpha$ is possible.

\item $T$ is called \emph{$<\! \delta$-closed} if $T$ has no branch of length $< \! \delta$.

\item
A non-empty tree $T \subseteq N(\vec \kappa)$ is called \emph{perfect} (in $N(\vec \kappa)$) if it is $<\! \delta$-closed and ever-branching.

\item
A non-empty tree $T \subseteq N$ is called \emph{non-stopping} (in $N(\vec \kappa)$) if for each node $t \in T$  there is a cofinal branch $b $ such that $t \in b$.

\item
A non-empty tree $T \subseteq N$ is called \emph{pre-perfect} (in $N(\vec \kappa)$) if it is ever-branching and non-stopping.

\end{enumerate}
\end{definition}

The property of being ever-branching is not influenced by the choice of the increasing cofinal sequence $\langle\kappa_\alpha \such \alpha < \delta \rangle$. We use these notions for various ordinals $\delta$.

Our predominating cardinal arithmetic assumption in the context of perfect tree forcings is as follows:

\begin{definition}\label{delta_strong_limit} We say that $\kappa$ is a $\delta$-\emph{strong limit} if for all $\tau<\kappa$, $\tau^\delta < \kappa$.
\end{definition}

\subsection{Versions of Tree Forcing}

We fix $\vec{\kappa} = \langle \kappa_\alpha \such \alpha < \lambda\rangle$ for a strictly increasing sequence of regular cardinals with $\omega \leq \lambda$ and let $\kappa = \sup (\vec{\kappa})$. In this paper we would like to cover a wide range of perfect tree forcings $\bP_\kappa$ involving trees $p \subseteq N(\vec{\kappa})$ of height $\lambda$ without being completely encyclopedic. For this purpose, we will cover three main schemes in which tree forcings might vary:



\begin{enumerate}
\item \emph{The type of tree: Miller versus Laver.} Miller-style forcings allow for fairly chaotic branching behavior, whereas Laver-style forcing allows a stem after which all nodes are splitting.


\item \emph{The degree of closure of the trees that serve as conditions}. This is regarding the closure of the trees themselves as well as the closure of splitting nodes. 

\item \emph{A specification of the manner of splitting.} We consider variations of Condition \ref{criterion1}, meaning that we will sometimes ask that the splitting sets themselves take a certain form, possibly a club or a stationary set.

\end{enumerate}

Now we introduce precise definitions, beginning with the types of trees.

\begin{definition}\label{the_forcing} Let $\vec{\kappa}= \langle \kappa_ \alpha \such \alpha < \lambda\rangle$ be an increasing sequence of cardinals with limit $\kappa$ and $\cf(\kappa) = \lambda \geq \omega$.

Conditions in the forcing $\Miller$ are perfect trees in $N(\vec{\kappa})$. 
If the forcing $\Miller$ is defined with respect to a cardinal $\kappa$, we may use the notation $\Miller_\kappa$.

Subtrees are stronger conditions: $q \leq p$ if $q \subseteq p$. 

\end{definition}

\begin{remark} If $\kappa$ is any singular cardinal, then the poset $\M_\kappa$ is the poset $\bP_\kappa$ of Dobrinen,  Hathaway, and Prikry \cite{DobrinenHathawayPrikry} referred to in the abstract.\end{remark}

We recall a Laver-style relative $\Laver_\kappa$ of $\Miller_\kappa$ to which our results also apply.


\begin{definition}\label{more_forcings} Let 
$\vec{\kappa} = \langle \kappa_\alpha \such \alpha < \lambda \rangle$ be an increasing sequence of regular cardinals that converges to $\kappa$, and $\lambda < \kappa$.

Singular Laver-Namba forcing $\Laver_\kappa$ is defined as follows: Conditions are $(<\!\lambda)$-closed trees
$p \subseteq N$ with a unique stem $s$ such that 
\begin{equation}\label{1.2} \forall t \in p (t \sqsupseteq s \rightarrow
|\Succ_p(t)| = \kappa_{\dom(t)}).
\end{equation}

In $\Laver_\kappa$ subtrees are stronger conditions.
\end{definition}

Next, we give definitions regarding the closure of the trees that serve as conditions.

\begin{definition} If conditions are $(<\!\lambda)$-closed trees $p$, we just write $\Miller_\kappa$ or $\Laver_\kappa$. If we weaken perfect to pre-perfect, we denote these posets ${}^{\textup{pre}}\Miller_\kappa$ or ${}^\textup{pre}\Laver_\kappa$. \end{definition}


\begin{definition} If conditions are ``$(<\!\lambda)$-splitting-closed'' trees $p$, i.e., any increasing $< \lambda$-sequence $\langle t_\alpha \such \alpha < \delta\rangle$, $\delta < \lambda$, of $\kappa_\alpha$-splitting nodes $t_\alpha$ in $p$ has a limit (the union) which is a $\kappa_\delta$-splitting node in $p$. We write ${}^{\textup{scl}}\Miller$ for the $(<\!\lambda)$-splitting-closed version of $\Miller$. In the case of Laver forcings $(<\!\lambda)$-closure implies $(<\!\lambda)$-splitting closure, so there is no meaningful distinction.\end{definition}

For $\lambda = \omega$, the ``vertical'' requirements coincide and describe the trees without maximal nodes.

Finally, we give definitions regarding the manner of splitting.

\begin{definition} We enumerate the versions we consider below.
\begin{enumerate}[(i)]
\item We write $\Laver = \Laver^\textup{stat}$ if for any splitting node $t \in p \in \bP$, we have $\OSucc_p(t) $ is a stationary subset of $\{\beta \in \kappa_{\dom(t)} \such \cf(\beta) = \omega \}$. For $p \in \Miller^\textup{stat}$ we require that for any node there is an extension $t$ that splits into a stationary subset of 
$\{\beta \in \kappa_{\dom(t)} \such \cf(\beta) = \omega \}$.\footnote{$\Laver^\textup{stat}$ has been used to add exact upper bounds in the context of \textup{\textsf{PCF}} theory \cite{CummingsMagidor2010}. We are not aware of any use of $\Miller^\textup{stat}$ yet. We will neglect the small adjustments that need to be made to some our definitions that would tailor them for $\Miller^\textup{stat}$.} 

This is the case of $I$-positive sets for an ideal $I$ whose additivity matters. 
\item
We write $\Laver = \Laver^\textup{club}$ if $\OSucc_p(t)$ contains a club subset of $\kappa_{\dom(t)}$.  For $p \in \Miller^\textup{club}$ we require that for any node there is an extension $t$ that splits into a superset of a club in $\kappa_{\dom(t)}$. Here splitting is into a highly closed filter.
\item
 If splitting is defined over a cardinal-wise criterion as in Clauses~\eqref{criterion1} or \eqref{1.2} we write just $\Miller$ and $\Laver$.
 \end{enumerate}\end{definition}

\begin{remark}
In the case of countable $\lambda$, the  forcing $\Laver^\textup{stat}_\kappa$ is the one studied by Cummings and Magidor \cite{CummingsMagidor2010}.
\end{remark}

Each version can be combined in various ways with each of the three named splitting criteria, e.g. ${}^{\textup{pre}}\Laver_\kappa^\textup{stat}$, $\Laver_\kappa^\textup{club}$.

\begin{observation}\label{closure_of_the_forcing}
The forcing notions ${}^{\textup{scl}}\Miller^\textup{club}_\kappa$ and $\Laver_\kappa^\textup{club}$ are $(< \! \lambda)$-closed.
\end{observation}

\begin{proof} We show that for a decreasing sequence $\langle p_\alpha \such \alpha < \delta \rangle$ of conditions of length $\delta < \lambda$,
the intersection $q = \bigcap_{\alpha<\delta} p_\alpha$ is a lower bound. 
In the Miller case, we have to show that $q \in {}^{\textup{scl}}\Miller_\kappa^\textup{club}$. 
The tree $q$ is not empty, because $\emptyset \in p_\alpha$ for each $\alpha < \delta$. Since any $p_\alpha$ is a $(<\!\lambda)$-closed tree, $q$ is also a $(<\!\lambda)$-closed tree. Now we show that $q$ is ever-branching.   Given $s \in q$, we choose an $\sqsubset$-increasing  sequence 
$\langle t_\alpha \such \alpha < \delta \rangle$ such that $t_\alpha \in \Split_\alpha(p_\alpha)$ is above $s$. We let $t_\delta = \bigcup\{t_\alpha \such \alpha < \delta\}$. Then for any $\alpha < \delta$, the limit $t_\delta$ is a $(\geq \kappa_\delta)$-splitting node in $p_\alpha$, because $p_\alpha$ fulfills that any $(<\!\lambda)$-limit of splitting nodes is a splitting node of the fitting degree. Now we use $\delta < \lambda \leq \kappa_\delta$ and the $(<\! \kappa_\delta)$-completeness of the club filter in order to intersect $\OSucc_{p_\alpha}(t_\delta)$, $\alpha < \delta$, and find that $t_\delta$ is a $\kappa_\delta$-splitting node in $q$. The proof for club Laver forcing $\Laver_\kappa^\textup{club}$ is simpler.
\end{proof}

In the literature there are analogs of this closure for regular $\kappa$ (in this case $\lambda = \kappa$): Baumgartner's higher Silver forcing \cite[Section 6]{baumg:a-d}, Kanamori's seminal work \cite{Kanamori-higher-sacks} on forcings with higher Sacks trees, Landver's \cite{Landver-1992} versions of higher Sacks trees, and club $\kappa$-Miller forcing, e.g., in \cite{friedman-zdomskyy}.

Observation~\ref{closure_of_the_forcing} is sharp in the case of uncountable $\lambda$.
We will prove in Section 4 that the other named variations are not $\omega$-distributive. Waiving the ``vertical condition'' that any $(<\!\lambda)$-limit of splitting nodes is a splitting node (see Subsection~\ref{SubS4.1}) or
 waiving the closure of set of ordinal splitting sets (see Subsections~\ref{SubS4.2}, \ref{SubS4.3}) allows us to show the failure of $\omega$-distributivity. All pre-perfect versions strongly violate $\omega$-distributivity (see Subsection \ref{SubS4.4}).
For uncountable $\lambda$, $(<\!\lambda)$-closure implies $\omega$-distributivity.
Hence in the uncountable case, the versions in Observation~\ref{closure_of_the_forcing} 
are the only $(<\!\lambda)$-closed posets among the named forcings.

\subsection{Notable Features of Perfect Tree Forcing}\label{SubS1.4}

A notable feature of perfect tree forcings in general is that of minimality, which motivated Sack's early work \cite{Sacks:perfect}. For our purposes, we will consider a relaxed version of minimality.

\begin{definition}[\cite{Sacks:perfect}] If $\lambda$ is a cardinal, a forcing $\bP$ is \emph{minimal for $\lambda$-sequences} if for any transitive model $M$ such that $V \subseteq M \subseteq V[\bP]$, if $A \in M$ is a function $A: \lambda \to \On$ such that $A \notin V$, then $M = V[\bP]$.\end{definition}

In the case of countable cofinality, minimality is often understood through the rendering of the generic branch. This is also justified in the higher cases where we are dealing with $(<\!\lambda)$-closed notions of forcing. 

 \begin{proposition}\label{generic-branch-prop}
 Let $\bP$ be ${}^{\textup{scl}}\Miller_\kappa^\textup{club}$ or $\bP= \Laver^\textup{club}$.
Then the so-called generic branch
\[ g_G=\bigcup \{\Stem(p) \such p \in G\}
\]
is an element of $\prod_{\alpha < \lambda} \kappa_\alpha$ and $V[g_G] = V[G]$.
\end{proposition}

\begin{proof}
The $(<\!\lambda)$-closure of the forcing notion 
implies that $g_G \in \prod_{\alpha < \lambda} \kappa_\alpha$. Now we show $V[G] \subseteq V[g_G]$. 
Suppose $q \Vdash g_G \in [p]$. Then $q \leq p$, since otherwise we could pick $t \in q \setminus p$ and have $q_t \Vdash g_G \not\in [p]$. We claim
$G = \{p \in \bP \such g_G \in [p] \}$. It is obvious that $p \in G$ implies $g_G \in [p]$. Conversely, suppose that $g_G \in [p]$. Then there is some $q \in G$ such that $q \Vdash g_G \in [p]$. We have $q \leq p$ and thus $p \in G$.
\end{proof}

In all of our minimality results, we will work with the generic branch. For $\lambda = \omega$, the closure requirement is vacuous. For uncountable $\lambda$, we will show that the closed versions are non-minimal, as is at least one of the non-closed versions (see Theorem~\ref{aleph1-embedding} below).

Another notable property of tree forcings is that of fusion.

\begin{definition}\label{fusion_stuff_usual} 
 We fix a sequence $\langle\kappa_\alpha \such \alpha < \lambda\rangle$ of regular cardinals converging to $\kappa$. The notions $\Split_\alpha(p)$ and $\leq_\alpha$ depend on this choice. Let $\bP$ be any of our forcings. 
 For convenience, we restrict all our forcing notions to the dense sets of conditions in which $\Stem(p)$ is $\kappa_0$-splitting. 
 We write $\lambda = \cf(\kappa)$.
\begin{enumerate} 
\item Take $p \in \bP_\kappa$. By induction on $\alpha \in \lambda$, we define
\begin{equation*}
\begin{split}
\Split_\alpha(p) = \bigl\{t \in p 
\such & t \mbox{ is $\sqsubseteq$-minimal in $p$ such that }\\
&|\Succ_p(t)| \geq \kappa_\alpha
\wedge 
\exists  \langle t_\beta \such \beta < \alpha \rangle,
\\
&
\bigl(\forall \beta < \gamma < \alpha, t_\beta \sqsubset t_\gamma \sqsubset t, 
\\
& \wedge \forall \beta < \alpha, t_\beta \in \Split_\beta(p)\bigr)
\bigr\}.
\end{split}
\end{equation*}
Observe that $\Split_\alpha(p)$ is an antichain in $p$. In particular, for $p \in \Laver_\kappa$, we have $\Split_\alpha(p) = p(\dom(\Stem(p)) + \alpha)$ and that for $t \in \Split_\alpha(p)$, $\OSucc_p(t) \subseteq \kappa_{\dom(\Stem(p))+\alpha}$ is of cardinality
$\kappa_{\dom(\Stem(p))+\alpha}$.
\item Let $p, q \in \bP$, $\alpha < \lambda$. We let $q \leq_\alpha p$ if $q \leq p$ and $\Split_\alpha(p) = \Split_\alpha(q)$.

\item 
A sequence $\langle p_\alpha \such \alpha < \delta \rangle$ such that $\delta \leq \lambda$ and for $\alpha <\gamma < \delta$, $p_{\gamma} \leq_\alpha p_\alpha$ is called a \emph{fusion sequence}.
\end{enumerate}
\end{definition}

We recall the so-called fusion lemma:

\begin{lemma}\label{fusion_lemma} Suppose that either $\cof(\kappa) = \lambda = \omega$ and that $\bP$ is any of the perfect versions or that $\cof(\kappa)=\lambda > \omega$ and that $\bP$ is one of the closed versions, $\Laver_\kappa^\textup{club}$ or $\,{}^{\textup{scl}}\Miller_\kappa^\textup{club}$. Then for a fusion sequence $\langle p_\alpha \such \alpha < \delta\rangle$, $1\leq\delta \leq \lambda$, we have $q_\delta = \bigcap\{p_\alpha \such \alpha < \delta\} \in \bP$ and $q_\delta \leq_\alpha p_\alpha$ for all $\alpha < \delta$.
\end{lemma}

Considerations of fusion become more complicated when we consider the cases of higher cofinality. In particular, the fusion lemma implies some $(<\!\lambda)$-closure:
 For this look at a limit $\delta < \lambda$ and let $\langle p_\alpha \such \alpha < \delta\rangle$ be given. If $q=\bigcap \{p_\alpha \such \alpha < \delta\}$ is a condition, then $\Split_\delta(q) \neq \emptyset$, and for any $t \in \Split_\delta(q)$ we have that $q \restriction t $ is a lower bound of
$\langle p_\alpha \rest t \such \alpha < \delta\rangle$; the latter is just an ordinary $\leq$-descending sequence. In Section~\ref{balcar-simon-sec} we will explore situations in which closure of the forcing cannot be assumed.

\section{Improvements to the Countable Case}\label{countable-sec}

\subsection{A Short Proof of Prikry's Theorem}

We begin by giving a proof of Prikry's theorem (see \cite[Theorem 3.5]{DobrinenHathawayPrikry}) that harmonizes with the proof from Jech's textbook \cite{Jech3} that under \CH\ the classical Namba forcing is $(\omega,\omega_1)$-distributive. However, we do not use a rank function. This approach would also work for the classical Namba forcing.


\begin{theorem}[Prikry]\label{prikry-distributivity} 
Suppose that $\cf(\kappa) = \omega$, and assume $\forall \nu < \kappa$, $\nu^\omega < \kappa$. Then for any $\nu < \kappa$, $\Miller_\kappa$ is $(\omega, \nu)$-distributive.
\end{theorem}


The following proposition is what allows us to avoid a rank function.

\begin{proposition}\label{imperfect_antichains}  Suppose that $\kappa$ is the supremum of $\vec \kappa = \langle \kappa_\alpha \such \alpha < \delta \rangle$ where $\delta$ is an ordinal less than $\kappa$ and let $T \subset N(\vec \kappa)$ be a tree that has no perfect subtree. Then there is a maximal antichain $A= \langle t_\xi \such \xi < \Theta \rangle \subseteq T$ such that $\forall t \in A, \exists \alpha < \delta, \forall t' \sqsupseteq t, |\Succ_T(t')| < \kappa_\alpha$.\end{proposition}

\begin{proof}  We choose elements inductively: Suppose $\langle t_\xi \such \xi < \eta \rangle$ has been defined and the sequence is not already a maximal antichain. Then we can find some $t_\eta$ incomparable with $t_\xi$ for $\xi<\eta$ that has the desired property. Otherwise, the downwards closure of $\{s \in T \such \forall \xi<\eta, t_\xi \perp s\}$ would be a perfect subtree of $T$.\end{proof}

Now we can prove Theorem~\ref{prikry-distributivity}.

\begin{proof} Let $p \Vdash ``\dot f: \omega \to \nu$''. We will find some $q \le p$ and $g: \omega \to \nu$ such that $q \Vdash ``\dot f = \check g$''.

We define antichains $A_n$ for $n<\omega$ by induction on $n<\omega$. At the same time, we define sets $\{q_s :s \in \bigcup_{n<\omega}A_n\}$ and $\{\alpha_s : s \in \bigcup_{n<\omega}A_n\}$ by induction on the length of $s$ such that the following hold:

\begin{enumerate}[(i)]

\item All elements of $A_n$ lie below elements of $A_{n+1}$.

\item If $s_1 \sqsubseteq s_2$, then $q_{s_2} \le q_{s_1}$.

\item If $s \in A_n$, then $q_s \Vdash ``\dot{f}(n) = \alpha_s$''.

\item $s \sqsubseteq \Stem(q_s)$.

\item If $s \in A_n$, then $\Stem(q_s)$ is a $\kappa_n$-splitting node.

\end{enumerate}

The construction goes as follows: Let $q_{\emptyset} \le p$ be a condition deciding $\dot{f}(0)$ such that if $s = \Stem(q_{\emptyset})$ then $s$ is $\kappa_0$-splitting. Then let $A_0 = \{s\}$. Now suppose $A_n$ has been defined. Given $s \in A_n$, for each successor $t$ of $\Stem(q_s)$, choose $q_t \le q_s$ and $\alpha_t$ such that $t \sqsubseteq \Stem(q_t)$, $q_t \Vdash ``\dot{f}(n+1) = \alpha_t$'', and $\Stem(q_t)$ is a $\kappa_{n+1}$-splitting node. Then collect all such $t$'s and include them in $A_{n+1}$.

Let $q$ be the downwards closure of $\bigcup_{n<\omega}A_n$, which is a perfect tree and hence a condition in $\M_\kappa$. For any $g: \omega \to \nu$, we define $q(g) \le q$ (not necessarily perfect) as follows: If $\bar \beta = \langle \beta_0,\ldots,\beta_n \rangle \in \nu^{<\omega}$, let $q(\bar \beta) = \bigcup \{q_s: s \in A_n, \bar{\beta} = \langle \alpha_\emptyset,\ldots,\alpha_{s \rest k},\ldots,\alpha_s \rangle \}$ and $q(g) = \bigcap_{n<\omega}q(g \rest n)$. It is immediate that $q(g) \Vdash ``\dot f = \check g$'' \emph{if} $q(g)$ is a condition in $\M_\kappa$.

\textbf{Claim:} There is some $g: \nu \to \omega$ such that $q(g)$ is a perfect tree.

To prove the claim, suppose otherwise for contradiction. Then for each $g: \omega \to \nu$, there is an antichain $A_g \subset q$ given by Proposition~\ref{imperfect_antichains} such that for all $t \in A_g, \exists n_g^t$ such that $\forall t' \sqsupseteq t$, $t'$ is $< \! \kappa_{n_g^t}$-splitting.

Now construct a $\sqsubset$-increasing sequence $\langle s_n \such n < \omega \rangle$ where $s_n \in A_n$ for each $n$ as follows: Make arbitrary choices until we have $\kappa_{n} > \nu^\omega$. Then suppose we are given $s_n$ and let $t =\Stem(q_{s_n})$. Then the set
\begin{equation*}
 \{s \in \Succ_q(t) \such \exists g \colon \omega \to \nu, t \in A_g, t \sqsubseteq s, n_g^t \le n \}
 \end{equation*}
\noindent has cardinality $< \! \kappa_n$ while $|\Succ_q(t)|=\kappa_n$. So choose $s_{n+1} \in A_{n+1}$ outside of this set.

Now let $g$ be given by $g(n) = \alpha_{s_n}$ for $n<\omega$. There is some $t \in A_g$ and some $n$ such that $s_n \sqsupseteq t$. But we have a contradiction when we consider $s_m$ such that $m \ge n, n_g^t$. Hence we have proved the claim.

If $g: \nu \to \omega$ witnesses the claim, then $q(g) \in \M_\kappa$ and $q(g) \Vdash ``\dot f = \check g$.\end{proof}

This style of argumentation works for other tree forcings for singulars of countable cofinality. In particular, we obtain a new proof of the same result for the forcing $\Laver_\kappa^\textup{stat}$ studied by Cummings and Magidor. Previous arguments, like the ones in \cite[Ch.~XI, Section 3, especially 3.3., 3.4 and 3.6]{Sh:b} and \cite[Theorem 2.3]{BC},  used an intermediate step in which regular trees are sought (see Theorem 2.2 in \cite{BC}).
A tree $p$ is called regular if there is a set $\{\ell_n \such n < \omega \}\subseteq \omega$, such that for any $n$, any $t \in p(\ell_n)$ is $\kappa_n$-splitting.

\begin{theorem}\label{Laver_namba} 
If $\cf(\kappa) = \omega < \kappa$ and $\kappa$ is an $\omega$-strong limit, then $\Laver_\kappa^\textup{stat}$ is $(\omega,\nu)$-distributive for all $\nu<\kappa$.\end{theorem}

We will use an important aspect of the Cummings-Magidor arguments, which to some extent derives from Laver's original forcing \cite{Laver-1976}.

\begin{fact}[See, e.g., {\cite[Fact one, page 3341]{CummingsMagidor2010}}]\label{strong-decision} If $\cf(\kappa) = \omega$, $\Laver_\kappa^\textup{stat}$ is defined on the product $N(\vec \kappa)$, $p \in \Laver_\kappa^\textup{stat}$ has a stem of length at least $n$, and $p$ forces that $\dot \gamma$ is an ordinal below $\mu$ for some $\mu<\kappa_n$, then there is some $q \le_0 p$ such that $q$ decides a value for $\dot \gamma$.\end{fact}

\begin{proof}[Proof of Theorem~\ref{Laver_namba}] Suppose $p \Vdash ``f : \omega \to \nu$'' and use Fact~\ref{strong-decision} to assume without loss of generality that $p$ has a stem of length $m$ where $\nu^\omega<\kappa_m$ and $p$ decides $\dot{f}(n)$ for all $n \le m$.

We define a a condition $q \le p$ as well sets $\{q_s : s \in q \}$ and $\{ \alpha_s : s \in q \}$ with the following properties:

\begin{enumerate}[(i)]

\item $\forall s$, $q_{s {}^\frown \langle \alpha \rangle} \le_0 q_s \rest
 (s {}^\frown \langle \alpha \rangle)$.

\item If $s_1 \sqsubseteq s_2$ then $q_{s_2} \le q_{s_1}$.

\item If $|s|=n$ then $q_s \Vdash \dot{f}(n) = \alpha_s$.

\end{enumerate}

This construction is a more controlled analog of the one in the proof of Theorem~\ref{prikry-distributivity}, and it works using Fact~\ref{strong-decision}. We also define $q(g)$ for $g: \omega \to \nu$ as in the proof of Theorem~\ref{prikry-distributivity}. The condition $q \le p$ will be the union of all of the $s$'s.

We construct a tree $r \le q$ by induction on the length of nodes $t$ as follows: $r$ has the same stem as $q$, and if we have established that $t \in r$ and $|t| = n$, then include $t {}^\frown \langle \alpha \rangle \in r$ if and only if $\alpha$ is \emph{not} in the set
\begin{align*}
\{\beta < \kappa_{n} : & \exists g : \omega \to \nu, t \in q(g) \text{ but has a non-stationary}\\ & \text{set of successors in } q(g),\text{ and }\beta \in \OSucc_{q(g)}(t)\},\end{align*}
\noindent which is non-stationary.

We claim that there is some $g$ such that $q(g) \cap r$ is a condition in $\Laver_\kappa^\textup{stat}$. Suppose otherwise. Then for every $t \in r$, $t$ fails to be a stem of $q(g) \cap r$. This means that there is some $s \sqsupseteq t$ such that $|s|=n$ for some $n$ and the set of successors of $s$ in $q(g) \cap r$ is non-stationary in $\kappa_{n}$, and so $\Succ_{(r \rest s) \cap q(g)}(s) = \emptyset$. For each $g:\omega \to \nu$, we thus build maximal antichains $A_g \subseteq r$ using the idea of Proposition~\ref{imperfect_antichains} such that for all $s \in A_g$, $\Succ_{(r \rest s) \cap q(g)}(s) = \emptyset$. Finally, let $b$ be any cofinal branch of $r$ and let $h : n \mapsto \alpha_{b \rest n}$. Then there is some $t \in b$ and some $s \in A_h$ such that $s \subseteq t$, which is a contradiction because this implies that $t \notin r$.\end{proof}

\subsection{Minimality Without Measurables}
\label{Sub2.2}

As before, $\bP$ is $\Miller_\kappa$ where $\kappa$ is a singular of countable cofinality. We  prove the following theorem, which waives the assumption in Theorem 6.6 from \cite{DobrinenHathawayPrikry} that the
$\kappa_n$'s are measurable.


\begin{theorem}\label{waiving_meas} 
Suppose $p \in \bP$ and $p \Vdash \dot{A} \colon \omega \to \check{V} \wedge \dot{A} \not \in \check{V}$. Then $p \Vdash ``\Gamma(\bP) \in \check{V}[\dot{A}]$''. In other words, $\bP$ is  minimal for $\omega$-sequences.
\end{theorem}

Here we use $\Gamma(\bP)$ to refer to the canonical $\bP$-name for a $\bP$-generic filter over $V$.

\begin{definition}[See {\cite[Definition 6.1]{DobrinenHathawayPrikry}} and {\cite[Lemma 11]{Brown-Groszek}}]\label{def_A-sharp} Let $\bP$ be either $\Miller_\kappa$ or $\Laver_\kappa$ for a singular $\kappa$. Let $\dot{A}$ be a $\bP$-name such that $1_{\bP} \Vdash \dot{A} \colon \omega \to \check{V} \wedge \dot{A} \not\in \check{V}$. For each condition $p \in \bP$ we let $\psi_p \colon p \to {}^{<\omega} V$ be the function that assigns to each node 
$t \in p$ the longest sequence $s = \psi_p(t)$ such that 
$p \rest t \Vdash \check{s} \sqsubseteq \dot{A}$.

Given $q \le p$, a splitting node $t \in \Split(q)$ is called \emph{$\dot{A}$-sharp} (with respect to $q$) if 
$\{\psi_p(c)\such c \in \Succ_q(t) \}$ consists of pairwise incompatible finite sequences. A condition $q \le p$ is $\dot A$-sharp if all of its splitting nodes are $\dot A$-sharp.
\end{definition}

\begin{lemma}[{\cite[Lemma 9]{Brown-Groszek}} or {\cite[Lemma 6.2]{DobrinenHathawayPrikry}}]\label{lemma_6.2_DHP} For $\bP = \M_\kappa$ or $\bP = \Laver_\kappa$ for $\omega = \cf(\kappa) < \kappa$, suppose that  $p \Vdash \dot{A} \colon \omega \to \check{V} \wedge \dot{A} \not \in \check{V}$ and that each splitting node of $p$ is $\dot{A}$-sharp.
Then $p \Vdash \Gamma(\bP) \in \check{V}[\dot{A}]$.
\end{lemma}

The basic idea of Lemma~\ref{lemma_6.2_DHP} is that the generic sequence, and therefore the generic filter, can be reconstructed from $\dot A_G$ using Proposition~\ref{generic-branch-prop} by determining a condition with the sharp coding.


Thanks to Lemma~\ref{lemma_6.2_DHP}, Theorem~\ref{waiving_meas} is implied by the following lemma.

\begin{lemma}\label{sharp_dense} Suppose that  $p \Vdash \dot{A} \colon \omega \to \check{V} \wedge \dot{A} \not \in \check{V}$. Then there are densely many $q \leq p$ in $\Miller_\kappa$ that are $\dot{A}$-sharp.
\end {lemma}

Now we state our main lemma.

\begin{lemma}\label{sharp-node-miller} Given $q \le p$, $n \in \omega$, and $s \in q$, there is some $t \sqsupseteq s$ and $q' \le q$ such that $\Stem(q') = t$, $t$ is $\kappa_n$-splitting, and $t$ is $\dot A$-sharp with respect to $q'$.\end{lemma}

\begin{proof} The procedure consists of the following five steps:

\ms

Step 1: We choose $t \sqsupseteq s$, $t \in q$ such that $|\Succ_{q}(t)| \geq \kappa_n$.

\ms

Step 2: We fix an injective enumeration
$\langle t_\alpha \such \alpha < \kappa_n\rangle $ of (possibly a subset) of $\Succ_{q}(t)$.
By induction on $\alpha < \kappa_n$ choose $\bar{p}_\alpha = \langle p_{\alpha,i} \such i < \omega \rangle$ and finite sets $\langle A_{\alpha, i } \such i < \omega \rangle $ with the following properties:

\begin{enumerate}
\item $p_{\alpha,i+1} \leq p_{\alpha, i}$ for each $i$.
\item $p_{\alpha,0} = q \rest t_\alpha$.

\item For each $i$, $p_{\alpha,i+1} \Vdash \rge(\dot A \rest i) = \check{A}_{\alpha,i}$.

\item $\bigcup_{i < \omega} A_{\alpha,i}
\not\subseteq \bigcup_{i < \omega, \beta < \alpha} A_{\beta,i}$.
\end{enumerate}

Condition (4) is possible since otherwise, if we defined $r = \bigcup_{\alpha<\kappa_n} q \rest t_\alpha$, we would have some $\alpha<\kappa_n$ such that $r \Vdash ``\rge(\dot{A}) \subseteq \bigcup_{i<\omega,\beta<\alpha}\check{A}_{\beta,i}$''. Since the cardinality of $\bigcup_{i<\omega,\beta<\alpha}A_{\beta,i}$ is less than $\kappa_n$, this would imply that $r \Vdash \dot{A} \in \check{V}$ by $(\omega, \kappa_n)$-distributivity, which contradicts our assumption that $\dot A$ is a new sequence.

\ms 

Step 3: Choose for $\alpha < \kappa_n$ some $a_\alpha \in \bigcup_{i < \omega} A_{\alpha,i} \setminus  \bigcup_{i < \omega, \beta < \alpha} A_{\beta,i}$.

\ms 

Step 4: Choose $m < \omega$ and $B \subseteq \kappa_n$, $|B| = \kappa_n$ such that $(\forall \alpha \in B)( a_\alpha \in 
\bigcup_{i < m} A_{\alpha,i}$. This is possible since $\cf(\kappa_n) > \omega$.

Now $\{A_{\alpha,m} \such \alpha \in B\}$ are pairwise distinct sets of size $m$.
\ms

Step 5: We let $q' = \bigcup_{\alpha \in B} p_{\alpha,m}$. Then $\Stem(q') = t$ and $t$ is $\dot{A}$-sharp with respect to $q'$.\end{proof}

Finally, we prove Theorem~\ref{waiving_meas}.

\begin{proof}
We use a fusion sequence and the fact that $\Miller_\kappa$ is $(\omega, \nu)$-distributive for any $\nu< \kappa$.

Observe that if $s \in q$ is $\dot A$-sharp, $q' \le q$, $s \in q'$, and $\Succ_q(s) = \Succ_{q'}(s)$, then $s$ remains $\dot A$-sharp.

We explain the construction assuming the claim by defining a fusion sequence $(p_n)_{n<\omega}$ as follows: Let $p_{0} = p$. Suppose that $n \geq 1$ and $p_{n-1}$ is constructed. We then construct $p_n$ such that $p_n \leq_{n-1} p_{n-1}$ and such that each $s^+ \in \Split_n(p_n)$ is $\dot{A}$-sharp. Explicitly, for all $t \in \Split_{n-1}(p_{n-1})$ and $s \in \Succ_{p_{n-1}}(t)$, we find $s^+ \sqsupseteq s$ and $q(t,s)$ witnessing Claim~\ref{sharp-node-miller}. Then we let 
\[
p_n = \bigcup\{q(t,s) \such t \in \Split_{n-1}(p_{n-1}), 
s \in \Succ_{p_{n-1}}(t)\}.
\]

Finally, we let $\bar p$ be the fusion of $(p_n)_{n<\omega}$ so that $\bar p$ witnesses the statement of the lemma. This completes the proof.
\end{proof}

\section{The Bernstein Technique}\label{bernstein-sec}

\subsection{Failures of $(\cf(\kappa),2)$-Distributivity and Collapsing $\kappa$ to $\cf(\kappa)$}
\label{Sub3.1}

In this subsection we show that $\Miller_\kappa$ and $\Laver_\kappa$ are not $(\cf(\kappa),2)$-distributive. This answers Question 3.6 of \cite{DobrinenHathawayPrikry}. The same result also applies to the respective $(<\!\!\cf(\kappa))$-closed versions of these forcings.

Throughout this section, $\kappa$ is a singular cardinal of uncountable cofinality $\lambda$. Let $\bP$ be any perfect version of $\M_\kappa$ or $\Laver_\kappa$ (we exclude the pre-perfect versions here).


\begin{theorem}\label{answer}
If $\cf(\kappa) = \lambda > \omega$ and $\kappa$ is a $\lambda$-strong limit, then $\bP$ is not
$(\lambda,2)$-distributive.
\end{theorem}

We show that our assumptions regarding cardinal arithmetic are sharp using a result that probably goes back to Solovay.

\begin{remark} If $\tau, \lambda < \kappa$ and $\tau^\lambda \geq \kappa$, then $\bP$ is not $(\lambda, \tau)$-distributive. If 
$\tau, \lambda < \kappa$, $\tau^\lambda \geq \kappa$,  and $\tau \leq \lambda$, then $\bP$ again is not $(\lambda, 2)$-distributive. 
\end{remark}

\begin{proof} Let $f \colon \kappa \to {}^\lambda \tau$ be an injective function, let $b \colon \lambda \to \lambda \times \lambda$ be a a bijective function, and write $b(\alpha) = (b_1(\alpha), b_2(\alpha))$.
Let $x \colon \lambda \to \kappa$ be the $\bP$-generic branch. Then
$\alpha \mapsto f(x(b_1(\alpha)))((b_2(\alpha))$ is a new function from $\lambda$ to $\tau$. The second line of the statement follows from a coding that we get if $\lambda^\lambda \ge \kappa$.
\end{proof}

We recall the classical result of Bernstein for completeness.

\begin{lemma}[The Existence of
 $\cA$-Bernstein Sets.] Let $\mu$ be an infinite cardinal and let $\cA \subseteq {\mathcal P}(\mu)$, be such that $|\cA| \leq \mu$ and such that for any  $A \in \cA$ we have $|A| = \mu$.
Then there is a set $\{B_\alpha \such \alpha < \mu \}$
with the following properties:
\begin{enumerate}
\item for each $\alpha < \mu$,  $B_\alpha \subseteq \mu$ and $|B_\alpha|=\mu$,
\item for $\alpha \neq \beta$, $B_\alpha \cap B_\beta = \emptyset$,
\item for all $A \in \cA$, for any $\alpha < \mu$, 
$A \cap B_\alpha \neq \emptyset$.
\end{enumerate}
\end{lemma}

The $B_\alpha$'s are sometimes called Bernstein sets for $\cA$.

\begin{proof}
We enumerate $\cA$ as $\langle A_\beta \such \beta < \mu\rangle$.  By induction on $\gamma < \mu$ we choose an element $a_{\beta,\gamma}\in \mu$ by induction on $\beta < \gamma$ such that
\begin{align*}
a_{\beta,\gamma} & \in  A_\beta \setminus \{a_{\beta',\gamma'}
\such (\beta' < \gamma' < \gamma) \vee (\beta' < \beta <\gamma' = \gamma)\}.
 \end{align*}
Then we define $D_\alpha = \{a_{\alpha,\beta} \such \alpha<\beta < \mu\}$.
Now $D_\alpha$, $\alpha < \mu$, is a disjoint refinement of $\cA$. That means
$D_\alpha \in [A_\alpha]^\mu$ and for $\beta \neq \alpha$, $D_\alpha \cap D_\beta = \emptyset$.
We enumerate each $D_\alpha$ as $\langle d_{\alpha,\beta} \such \beta < \mu\rangle$.
We let $B_\beta = \{d_{\alpha, \beta} \such \alpha < \mu \}$.
\end{proof}

We will use a simple lemma extending Proposition~\ref{imperfect_antichains} that helps us avoid the assumption of $\GCH$ in our proof of Theorem~\ref{answer}.

\begin{lemma}\label{perfect_equivalence} Suppose that $\kappa$ is the supremum of $\vec \kappa = \langle \kappa_\alpha \such \alpha < \delta \rangle$ where $\delta$ is an ordinal less than $\kappa$, and that $\tau^\delta \le \kappa$ for all $\tau < \kappa$. Let $T \subset  N(\vec \kappa)$ be a tree with no branches of length $<\delta$. Then the following are equivalent:

\begin{enumerate}
\item $T$ has a perfect subtree.
\item The cardinality of $[T]$ is equal to the cardinality of $\prod_{\alpha<\delta}\kappa_\alpha$.
\item The cardinality of $[T]$ is greater than $\kappa$.
\end{enumerate}\end{lemma}

\begin{proof} First we show that (1) implies (2). It is enough to define an injection $\prod_{\alpha<\delta}\kappa_\alpha \hookrightarrow [T'] \subset [T]$ where $T'$ is a perfect subtree of $T$. First we will define an injection $\bar \varphi \colon N(\vec{\kappa}) \to T'$ by induction on $\beta$. We let $\bar{\varphi}(\emptyset)$ be a $\kappa_0$-splitting node of $p$. For the successor case, suppose that $\bar \varphi$ is defined on $\prod_{\alpha<\gamma}\kappa_\alpha$ and $\beta = \gamma+1$. Then for all $f \in \prod_{\alpha<\gamma}\kappa_\alpha$, find a $\kappa_\gamma$-splitting node $s \sqsupseteq \bar \varphi(f)$, enumerate $\Succ_{T'}(s)= \langle t_\xi \such \xi < \kappa_\gamma \rangle$, and then define $\bar \varphi(f^\frown \langle \gamma,\xi \rangle)$ to be $t_\xi$ for all $\xi<\kappa_\gamma$. The argument is similar if $\beta$ is a limit, but for $f \in \prod_{\alpha<\beta}\kappa_\alpha$ we must find the appropriate splitting node above $\langle\bar \varphi(f \rest \gamma) \such \gamma < \beta \rangle$. Finally, we define $\varphi : \prod_{\alpha<\delta}\kappa_\alpha \to [T']$ by letting $\varphi(f)$ be the branch given by $\langle\bar \varphi(f \rest \beta) \such \beta < \delta \rangle$.

The fact that (2) implies (3) is because of K{\"o}nig's Theorem.

To prove that (3) implies (1), suppose contrapositively that $T$ does not have a perfect subtree. Let $A \subset T$ be the antichain given by Proposition~\ref{imperfect_antichains}. Observe that since $A \subset T$ and $|T| \le \kappa$ (using that $\tau^\delta \le \kappa$ for all $\tau<\kappa$), it follows that $|A| \le \kappa$. If $t \in A$ and $\alpha$ is such that $\forall t' \sqsupset t, |\Succ_T(t')|<\kappa_\alpha$, it follows that there are no more than $\kappa_\alpha^\delta$-many branches containing $t$, meaning at most $\kappa$-many (again using that $\tau^\delta \le \kappa$ for all $\tau<\kappa$). Therefore, $T$ has at most $\kappa$-many branches.\end{proof}

The following can be seen as a variation on Silver's famous theorem for singular cardinals of uncountable cofinality.

\begin{lemma}\label{silver_counting} Let $\kappa$ be a singular $\lambda$-strong limit cardinal of uncountable cofinality $\lambda$. Let $\langle\kappa_\alpha \such \alpha < \lambda \rangle$ be a sequence of regular cardinals converging to $\kappa$, and for $\alpha<\lambda$ let $\mu_\alpha$ be the cardinality of $\prod_{\beta<\alpha}\kappa_\beta$.

If $T$ is a perfect tree in the space $\prod_{\alpha < \lambda}\kappa_\alpha$, then the set 
\[
\{ \alpha < \lambda \such |[T \nrest \alpha]| = \mu_\alpha \}
\]
\noindent contains a club.\end{lemma}

\begin{proof} Let $\langle\delta_\alpha \such \alpha < \lambda \rangle$ be the club such that for all $\alpha<\lambda$, $\delta_{\alpha+1} = \kappa_{\alpha+1}$, and such that for all limits $\alpha<\lambda$, $\delta_\alpha := \sup_{\beta<\alpha}\kappa_\alpha$. There is a club $D \subseteq \lambda$ such that for all $\alpha \in D$, $\alpha$ is a limit and $\delta_\alpha$ is a $\lambda$-strong limit. Observe that for all $\alpha<\lambda$, $T \nrest \alpha$ has no branches of length less than $\alpha$. Then observe that by Lemma~\ref{perfect_equivalence}, if $\alpha \in D$, then $|[T \nrest \alpha]| \ne \mu_\alpha$ implies that $|[T \nrest \alpha ]| \le \delta_\alpha$. Therefore, if the conclusion of the lemma we are trying to prove is false, then $S:= \{\alpha \in D \such |[T \nrest \alpha ]| \le \delta_\alpha \}$ is stationary in $\lambda$.

Now we can demonstrate the basic idea behind our definition of $S$. For all $\alpha \in S$, enumerate $[T \nrest \alpha]$ as $\langle t^\alpha_\xi \such \xi < \delta_\alpha \rangle$. If $b \in [T]$ is a cofinal branch, then for $\alpha \in S$ let $F(\alpha)$ be the least $\beta<\alpha$ such that there is some $\xi<\delta_\beta$ such that $b \nrest \alpha = t^\alpha_\xi$. This is a regressive function on $S$, so Fodor's Lemma implies that there is a stationary subset $S' \subseteq S$ and some $\beta<\lambda$ such that for all $\alpha \in S'$, $b \rest \alpha = t_\xi^\alpha$ for some $\xi<\delta_\beta$.

We can use this idea to show that there are at most $\kappa$-many cofinal branches of $T$. Suppose for contradiction that there are at least $\kappa^+$-many cofinal branches of $T$. Then because $2^\lambda < \kappa$, there is a stationary subset $S' \subset S$ and some $\beta < \lambda$ such that are $\kappa^+$-many branches $b$ with the property that for all $\alpha \in S'$, $b \nrest \alpha = t^\alpha_\xi$ for some $\xi<\delta_\beta$. Then since $\delta_\beta \le \kappa_\beta$ and $\kappa_\beta^\lambda < \kappa$, it follows that there is a function $h: S' \to \delta_\beta$ such that for $\kappa^+$-many distinct branches $b$, $b \nrest \alpha = t^\alpha_{h(\alpha)}$ for all $\alpha \in S'$. But these branches are determined cofinally, so this is impossible.

Since $T$ has at most $\kappa$-many branches, it cannot be a perfect tree.\end{proof}

Now we turn to the proof of Theorem~\ref{answer}:

\begin{proof} Again let $\mu_\alpha$ denote $|\prod_{\beta < \alpha} \kappa_\beta|$. For a limit $\alpha<\lambda$, let $\cA(\alpha)$ be the collection of subsets $X$ of $\prod_{\beta < \alpha} \kappa_\beta$ of size $\mu_\alpha$ (i.e$.$ the subsets with ``full'' size) that are the branches of a perfect tree in $\bigcup_{\gamma<\alpha} \prod_{\beta < \gamma} \kappa_\beta$. Note that there are at most $2^{\delta_\alpha} = 2^{\lim_{\beta< \alpha} \kappa_\beta} = \mu_\alpha$-many 
perfect trees in $\bigcup_{\gamma<\alpha} \prod_{\beta < \gamma} \kappa_\beta$ and hence $|\cA(\alpha)| \leq \mu_\alpha$.  
Let $\{B_{\alpha,\gamma} \such \gamma < \mu_\alpha\}$ be a 
set of pairwise disjoint  $\cA(\alpha)$-Bernstein sets. This means that if $X \subset \prod_{\beta < \alpha} \kappa_\beta$ is in $\cA(\alpha)$, then for $\gamma < \mu_\alpha$, $B_{\alpha,\gamma} \cap X \ne \emptyset$ and $X \cap (\prod_{\beta < \alpha} \kappa_\beta \setminus B_{\alpha,\gamma}) \ne \emptyset$. For $\alpha=0$ and successors $\alpha$, we let $B_{\alpha,\gamma} = \{0\}$.

Now we define a $\bP$-name for a function from
$\Lim(\lambda)$, the set of limit ordinals in $\lambda$,  to $2$. We define $\dot f$ as the following name:
\begin{equation}\label{binary-bernstein-name} 
\begin{aligned} 
 \{ \langle & \check{(\alpha, 0)} , p \rangle \such
\alpha \in \Lim(\lambda), \dom(\Stem(p)) \geq \alpha, \Stem(p)\nrest \alpha \in B_{\alpha,0}\} \cup 
\\ 
& 
\{\langle \check{(\alpha, 1)}, p \rangle \such
\alpha \in \Lim(\lambda), \dom(\Stem(p)) \geq \alpha, \Stem(p) \nrest  \alpha \not\in B_{\alpha,0}\}.
\end{aligned}
\end{equation}

Assume for contradiction that there is some $g \in {}^{\Lim(\lambda)} 2 \cap V$ and there is some $p \in \bP$ such that $p \Vdash \dot{f} = g$.

By Lemma \ref{silver_counting}, there is some limit ordinal $\alpha > \dom(\Stem(p))$ such that $|[p \nrest \alpha]| = \mu_\alpha$, and therefore $[p \nrest \alpha] \in \cA(\alpha)$. If $g(\alpha) =1$, then we strengthen $p$ to some condition $q$ such that
that $\Stem(q) \nrest \alpha \in B_{\alpha,0}$. This is possible, since $B_{\alpha,0}$ is a $\cA(\alpha)$-Bernstein set and $[p \nrest \alpha] \in \cA(\alpha)$, so $[p \nrest \alpha] \cap B_{\alpha,0} \neq \emptyset$.
If $g(\alpha) =0$, then we strengthen $p$ to some condition $q$ such that $\Stem(q) \nrest \alpha \not\in B_{\alpha,0}$. This is again possible, since $\{B_{\alpha,\gamma}\such \gamma < \mu_\alpha\}$ is a set of disjoint $\cA(\alpha)$-Bernstein sets and $[p \nrest \alpha] \in \cA(\alpha)$, so $[p \nrest \alpha] \not\subseteq B_{\alpha,0}$.

So in any case we have $q \leq p$ and $q \Vdash \dot{f}(\alpha) \neq g(\alpha)$, and hence a contradiction. Thus, forcing with $\bP$ adds a new function $\dot{f}_G \in {}^{\Lim(\lambda)} 2$ and hence $\bP$ is not $(\lambda, 2)$-distributive.\end{proof}

This Bernstein argument applies equally to the
Laver versions of $\bP$ (see Def.~\ref{more_forcings}(2))
and to the versions splitting into particular filter sets and also to the versions in which each limit of splitting nodes is a splitting node. Of course, Lemma~\ref{silver_counting} is not necessary for the Laver versions.

The Bernstein technique shows us that Lemma~\ref{silver_counting} leads to a stronger negation of
$\cf(\kappa)$-distributivity: there is a name for a collapsing function.

\begin{theorem}\label{lambda_onto_kappa}
If $\cf(\kappa) = \lambda > \omega$ and $\kappa$ is a $\lambda$-strong limit, then $\bP$ collapses $\kappa$ to $\lambda$.
\end{theorem}

\begin{proof}  For each $\alpha < \lambda$ we choose
a set $\cA(\alpha)$ and a set $\{B_{\alpha, \beta} \such \beta < \mu_\alpha \}$ of disjoint Bernstein sets as above.  
We define a $\bP$-name for a function from
$\Lim(\lambda)$ onto $\kappa$:
\begin{align*}
 \dot{f} = & \{\langle\check{(\alpha, \gamma)}, p \rangle \such
\alpha \in \Lim(\lambda), \gamma < \mu_\alpha, \dom(\Stem(p)) \geq \alpha, \Stem(p) \nrest \alpha \in B_{\alpha,\gamma}\}.
\end{align*}
The argument that $\Vdash_\bP ``\dot f : \Lim(\lambda) \twoheadrightarrow \kappa$'' is analogous to the one presented in the proof of Theorem~\ref{answer}.
For $\gamma < \kappa$ the set
\[
D_\gamma = \{p \in \bP \such \exists \alpha \in \Lim( \lambda),
p \Vdash \dot{f}(\check{\alpha}) = \check{\gamma} \}
\]
is a dense subset of $\bP$.
\end{proof}

Observe that the Bernstein technique does not work for pre-perfect trees.  In Section \ref{SubS4.4} we will show failure of $(\lambda, 2)$-distributivity for the pre-perfect versions using a different name.

\subsection{Non-Minimality and Cardinal Preservation for the $(<\!\cf(\kappa))$-Closed Versions}
\label{SubS3.2}

In this section we settle natural questions of minimality and cardinal collapses for ${}^{\textup{scl}}\M_\kappa^\textup{club}$ and $\Laver_\kappa^\textup{club}$ given singular $\lambda$-strong limits of cofinality $\lambda>\omega$. 

The first observation is a corollary to Theorem~\ref{lambda_onto_kappa}.
 For $\lambda=\kappa = \omega$, Groszek \cite[Theorem 5]{Groszek} proved an analogous result. For $\lambda = \omega < \kappa$, the technique of Brendle et al$.$ \cite[Proposition 77]{UncountableCichon} shows that 
  ${}^{\rm  scl}\M_\kappa^\textup{club}$ and $\Laver_\kappa^\textup{club}$ add an $\omega$-Cohen real and hence are not minimal for $\omega$-sequences.
Note that Theorem~\ref{waiving_meas} does not generalize to uncountable $\lambda$ because we do not have the $(\lambda,\nu)$-distributivity needed for Step 2 of Claim~\ref{sharp-node-miller}.

\begin{corollary}\label{closed-non-minimality} Suppose that $\bP$ is either ${}^{\rm  scl}\M_\kappa^\textup{club}$ or $\Laver_\kappa^\textup{club}$ and that $\kappa$ is a singular $\lambda$-strong limit of cofinality $\lambda>\omega$. Then $\Add(\lambda,1) \lessdot \bP$, and so $\bP$ is not minimal for $\lambda$-sequences.\end{corollary}

\begin{proof} This follows from two folklore lemmas. The first is that for homogenous forcings $\bQ_1$ and $\bQ_2$, there is a complete embedding $\iota: \bQ_1 \to \ro(\bQ_2)$ if and only if every filter $G$ that is $\bQ_2$-generic over $V$ induces a filter $H$ that is $\bQ_1$-generic over $V$. The second is that if $\bQ_1$ is $(< \! \lambda)$-closed and $\bQ_2$ is a filter such that $V[\bQ_2] \models ``|\mathcal{P}^V(\bQ_1)|=\lambda$'', then any $\bQ_2$-generic induces a $\bQ_1$-generic.

By Theorem~\ref{lambda_onto_kappa}, we have $V[\bP] \models ``|\mathcal{P}^V(\Add(\lambda,1))| =| (2^\lambda)^V|=\lambda$'', so we conclude that forcing with either ${}^\textup{scl}\M_\kappa^\textup{club}$ or $\Laver_\kappa^\textup{club}$ induces a generic on $\Add(\lambda,1)$.
\end{proof}

The rest of the section is dedicated to proving that the closed versions preserve $\kappa^+$ as well as the cardinals up to and included $\lambda$.


We prove a lemma on $\leq_\alpha$-stronger decision into boundedly-many possibilities. There are analogs used in many places, e.g.\ \cite{Kanamori-higher-sacks}.


\begin{theorem}\label{analogon_to_Kanamori}
Let $p \in \bP$ force that $\dot{f}$ be a $\bP$-name for a function from $\lambda$ into $\kappa^+$. Then there are a set $x$  of size at most $\kappa$  and a condition $q \leq p$ such that
\[
q \Vdash \rge(\dot{f}) \subseteq \check{x}.
\]
\end{theorem}

\begin{proof} First we establish a claim that works for any of the ever-splitting versions:

\textbf{Claim:} Let $p \in {}^{\textup{scl}}\Miller_{\kappa}^\textup{club}$ force that $\dot{\beta}$ is an ordinal and let $\alpha < \lambda$. Then there is some $q \leq_\alpha p$ and some set $x$ of size at most 
$|\Split_\alpha(p)| \cdot \kappa_\alpha$ such that $q \Vdash \dot{\beta} \in x$.


To prove the claim, consider $\alpha< \kappa$. By definition, $\Split_\alpha(p)$ is a maximal antichain. For each 
$s \in \Split_\alpha(p)$ we proceed as in the step $\alpha=0$, i.e.,  for each $t \in \Succ_p(s)$ we choose
some $q(t,s) \leq p \nrest t$ that decides $\dot{\beta}$ and we let
$x_s$ be a set containing all these decisions. We let 
$x = \bigcup\{x_s \such s \in \Split_\alpha(p)\}$. 
The size of $\Split_\alpha(p)$ is bounded by $\sup_{\beta < \alpha} \mu_\beta \leq \kappa$ by our assumptions on cardinal exponentiation in Definitions~\ref{the_forcing} and \ref{more_forcings}. 
Each $x_s$ has size at most $\kappa_{\dom(s)} < \kappa$.
Thus $|x| \leq |\Split_\alpha(p)| \cdot \kappa \leq \kappa$
We let 
\[
q = \bigcup\{q(t,s) \such s \in \Split_\alpha(p) \wedge t \in \Succ_p(s)\}.
\]
Then $q \leq_{\alpha} p$ and $q \Vdash \dot{\beta} \in \check{x}$. This proves the claim.


Given $p$, assume without loss of generality that $p \Vdash \dot f \colon \lambda \to \kappa^+$. Then construct a fusion sequence 
$\langle q_\alpha \such \alpha < \lambda \rangle$ such that 
there are $x_\alpha$ of size at most $\kappa$ such that 
$q_\alpha \Vdash \dot{f}(\alpha) \in x_\alpha$.

We choose the $q_\alpha$ be recursion on $\alpha$. In the successor steps we proceed according to the claim. In the limit steps $\alpha < \kappa$, we first let
$q'_\alpha = \bigcap\{q_\beta \such \beta < \alpha\}$.
By closure, $q'_\alpha$ is a condition. Now we perform a successor step to go from $q'_\alpha$ to $q_\alpha$.

The fusion limit $q = \bigcap \{q_\alpha \such \alpha < \lambda\}$ has the desired properties.
\end{proof}

Although the previous lemma applies to the closed version, there is another way of obtaining preservation of $\kappa^+$ for $\Laver^\textup{club}_\kappa$:

\begin{proposition} \label{chain} The singular club-Laver-Namba forcing $\Laver_\kappa^\textup{club}$ has the $\kappa^+$-chain condition.\end{proposition}

\begin{proof} Observe that two conditions in the club-Laver-Namba forcing are compatible if they have the same stem, and there are only $\kappa$-many stems.\end{proof}

Theorem~\ref{lambda_onto_kappa} on collapsing $\kappa$ to $\lambda$ pertains also to the $(<\!\lambda)$-closed versions of the forcings from this section.
Combining Theorem~\ref{analogon_to_Kanamori} with Theorem~\ref{lambda_onto_kappa} about collapsing
 $\kappa$ to $\lambda$
(for which we used $\lambda$-strong limits)  we get:

\begin{corollary}\label{harmless}
For  any $\lambda$-strong limit $\kappa$ and $\bP_\kappa$
being
 ${}^{\textup{scl}}\Miller_{\kappa}^\textup{club}$ or $\Laver_\kappa^\textup{club}$, we have
\[
\bP_{\kappa} \Vdash |\kappa| = \lambda \wedge \lambda^+ = (\kappa^+)^V
\]
and all cardinals $\leq\lambda$ are preserved.
\end{corollary}

\subsection{The Non-Miniminality of a Non-Closed Version}

The aim of this section is to show that we do not necessarily need closure to obtain non-minimality results.

\begin{theorem}\label{aleph1-embedding}  Let $\kappa$ be a singular cardinal of cofinality $\omega_1$. Let $\bP$ be any of $\Laver_\kappa$, $\Laver_\kappa^\textup{stat}$, and $\Laver_\kappa^\textup{club}$. Then $\Add(\omega_1,1) \lessdot \bP$.\end{theorem}

For the rest of this section we fix $\kappa$, a singular cardinal of cofinality $\aleph_1$. We will work with $\Laver_\kappa^\textup{}$ for simplicity.

\begin{definition} Let $\dot f$ be the binary function \ref{binary-bernstein-name} introduced in the proof of Theorem~\ref{answer}. Given $ p \in \Laver_\kappa$, let $b(p)$ be defined so that $\alpha \in \dom(b(p))$ and $b(p)(\alpha) = i$ if and only if $p \Vdash \dot f(\omega \cdot \alpha) = i$.\end{definition}

\begin{proposition} If $p \in \Laver_\kappa$, then $b(p) \in \Add(\omega_1,1)$.\end{proposition}

\begin{proof} This is the observation that if $\alpha$ is the length of the stem of $p$, then for all limit ordinals $\gamma$ in $(\alpha,\lambda)$, $p \rest \gamma$ is perfect and therefore $p$ does not decide $\dot{f}(\gamma)$ by the reasoning in the proof of Theorem~\ref{answer}.\end{proof}

\begin{lemma}\label{extensions} Suppose we are given $\bar p \in \Laver_\kappa$. If $\bar b = b(\bar p)$, then given any $b \in \Add(\omega_1,1)$ such that $b \le \bar b$, there is some $p \le \bar p $ such that $b(p) = b$.\end{lemma}

\begin{proof} Fix $\bar p$, $\bar b$, and $b$. Let $\alpha$ be the length of the stem $\bar s$ of $p$. Let $\delta = \dom(b)$. For $t \in \bar p$, we say that $t$ is \emph{good} if for all $\beta \in \Lim((\dom(t)+1)) \cap \delta$, $t \rest \beta \in B_{\beta,0}$ if and only if $b(\tilde \beta)=0$ where $\omega \cdot \tilde \beta = \beta$.

We will show by induction on limits $\gamma \in [\alpha+\omega,\delta]$ that:

\begin{quote} for all good nodes $s \sqsupseteq \bar s$ in $\bar p$, there is some $t \sqsupseteq s$ such that $t$ is good and $\gamma \in \dom(t)+1$.\end{quote}

The successor case of the induction pertains to ordinals of the form $\gamma+\omega$. Let $s \sqsupseteq \bar s$ be a good node of $p$ such that $\dom(s) = \bar \gamma < \gamma+\omega$. If $\gamma > \alpha$, apply the inductive hypothesis to find $s' \sqsupseteq s$ with domain $\gamma$. Because $p \in \Laver_\kappa$, $(p \rest s') \rest (\gamma+\omega)$ is perfect. Therefore we can find some $t \sqsupseteq s$ such that the set of predecessors of $t$ is a cofinal branch of $(p \rest s) \rest (\gamma + \omega)$ and $t \rest (\gamma+\omega) \in B_{\gamma+\omega,0}$ if $b(\tilde \gamma + 1)=0$ where $\omega \cdot (\tilde \gamma + 1) = \gamma+\omega$ and $t \rest (\gamma + \omega) \notin B_{\gamma+\omega,0}$ otherwise. Then $t$ is as sought.

Now suppose that $\gamma$ is a limit of limit ordinals, and in particular that $\langle\gamma_n \such n<\omega \rangle$ is a sequence of limit ordinals above $\alpha$ converging to $\gamma$. We pick a good node $s$ and define a perfect subtree $T$ of $(p \rest s) \rest \gamma$ by induction on $n$. Suppose we have define $T \rest \gamma_n$. Then for each cofinal branch $s'$ representing a cofinal branch of $T \rest \gamma_n$, choose $\kappa_{\gamma_n+1}$-many \emph{immediate} successors $s'_\xi$, $\xi<\kappa_{\gamma_n+1}$ of $s'$ in $p$ (again using that $p \in \Laver_\kappa$) which are still good nodes, and then choose good nodes $s''_\xi \sqsupseteq s_\xi$ such that $\gamma_{n+1} \in \dom s''_\xi$ for all $\xi<\kappa_{\gamma_n+1}$ using the inductive hypothesis. Then collect all such $s''_\xi$'s for each $s'$ representing a cofinal branch of $T \rest \gamma_n$ and let this form $T \rest \gamma_{n+1}$. Once we are done defining $T$, we choose a node $t$ of length $\gamma+1$ defining a cofinal branch of $T$ such that $t \rest \gamma \in B_{\gamma,0}$ if and only if $b(\tilde \gamma)=0$ where $\omega \cdot \tilde \gamma = \gamma$.

Finally, we choose some $t \in \bar p$ witnessing the statement for $\delta$ and we let $p = \bar p \rest t$.\end{proof}


We proceed with the proof of Theorem~\ref{aleph1-embedding}.

\begin{proof}
We show that there is a complete embedding $\iota : D \to \ro(\Laver_\kappa)$ where $D$ is a dense subset of $\Add(\omega_1,1)$. Specifically, let $D$ be the dense subset of $\Add(\omega_1,1)$ consisting of conditions $b$ such that $\dom(b) = \{\alpha < \gamma : \lim(\alpha)\}$ for some $\gamma<\omega_1$.
Given $b \in D$, define
\vspace{-3mm}
\begin{equation*}
\iota(b) := \{p \in \Laver_\kappa : \alpha \in \dom(b) \wedge b(\alpha) = i \Rightarrow p \Vdash ``\dot f(\omega \cdot \alpha) = i \text{''}\}.
\end{equation*}

Observe that $\iota(b) \ne \emptyset$ by Lemma~\ref{extensions}.

The basic properties of complete embeddings hold for $\iota$. We see that $\iota(b)$ is a regular open set, meaning that if $p \notin \iota(b)$, then there is some $q \le p$ such that $\{r \in \Laver_\kappa : r \le q \} \cap \iota(b) = \emptyset$. Moreover, $b' \le b$ implies that $\iota(b') \subseteq \iota(b)$ and $b' \perp b$ implies that $\iota(b') \cap \iota(b) = \emptyset$ by definition.

To prove completeness, suppose that $u \in \ro(\Laver_\kappa)$ and let $p \in u$. Let $b = b(p)$. We claim that $b$ witnesses completeness. If $b' \le b$, then let $p' \le p$ witness Lemma~\ref{extensions}. Then $\iota(b') \cap u \ne \emptyset$ because this intersection contains $p'$. \end{proof}

The fact that this construction cannot be generalized from Corollary~\ref{closed-non-minimality} leads us to our next section.

\section{The Balcar-Simon Technique and Failures of Countable Distributivity}
\label{balcar-simon-sec}

Here are some examples of failures of countable distributivity in the non-closed versions of forcings of perfect tree forcings. We demonstrate the ways in which an idea of Balcar and Simon can be used to obtain failures of countable distributivity  by finding a way to degrade a structure through countably many steps (see \cite{BalcarSimon88}, \cite{BalcarSimonHandbook}), mainly by picking out a non-stationary set and discarding limit points at each stage. All of these involve an ordering, and we cover $\prec^\textup{vert}$, $\prec^\textup{horiz}$, and $\prec^\textup{stat}$. For forcing with the pre-perfect versions we refer the reader to Subsection \ref{SubS4.4}.

\subsection{Slow Fusion}\label{Subs4.1new}

Each of our results in the secution will require a lemma stating that for each condition $p$ in the forcing we are interested in, there is some $q \prec p$ where $\prec$ is one of the orders mentioned above. These conditions $q$ will be provided by fusion sequences in which conditions are strengthened locally and only slightly. We call these ``slow fusion'' sequences. These will in the end be used to prove that the usual sequences in the non-closed versions are not viable because they break down at intermediate steps.

We describe the basic idea of slow fusion here. It is simpler in the Laver case, which by definition satisfies Kanamori's vertical requirement that limits of splitting nodes are splitting nodes. In this case, a $\le$-decreasing sequence $\langle p_\alpha \such \alpha < \lambda \rangle$ is slow if $p_\alpha$ differs from $\bigcap \{p_\beta \such \beta < \alpha\}$ only in the level $p_\alpha(\dom(\Stem(p_0)) + \alpha+1)$, and for any $\beta < \alpha$ and $t \in p_\alpha(\dom(\Stem(p_0)) + \alpha+1)$, we have $p_\alpha \rest t = p_\beta \rest t$. In other words, only one level is altered at a time.

The Miller case is analogous, but we must keep track of the antichains $\Split_\alpha(p)$ for $\alpha<\lambda$. In this case, in order to slowly descend from $\bigcap \{p_\beta \such \beta < \alpha\}$ to $p_\alpha$, 
a ``horizontal stripe'' of $\bigcap\{p_\beta \such \beta < \alpha\}$ that starts after $\Split_\alpha(\bigcap\{p_\beta \such \beta <\alpha\}$ and is bounded in the tree ordering is thinned out in the transition to $p_\alpha$, whereas
the other areas of $\bigcap\{p_\beta \such \beta <\alpha\}$ are just copied into $p_\alpha$. By the disjointness of the stripes shrinkage to a non-condition in an intersection over infinitely many previous steps shall be precluded.

The existence of the intermediate limit $\bigcap\{p_\beta \such \beta<\alpha\}$ is proved rigorously in Lemma~\ref{slow_fusion_lemma} below. The final intersection, meaning the eventual limit of the fusion sequence at stage $\lambda$, is justified in the same manner.

\begin{definition}\label{slow_fusion_stuff}
 We fix a sequence $\langle\kappa_\alpha \such \alpha < \lambda\rangle$ of regular cardinals converging to $\kappa$. Let $\bP$ be any of our forcings where $\lambda = \cf(\kappa)$.
 
 For $\alpha \leq \gamma$ we write $q \le_{\gamma,\alpha} p$ if the following hold:
 
 \begin{enumerate}
\item $q \leq p$,

\item $\Split_\alpha(p) = \Split_\alpha(q)$,

\item $\forall t \in \Split_\alpha(p), \Succ_{p}(t) = \Succ_{q}(t)$,

\item $\forall s \in \Split_{\gamma}(q), \forall t \in \Succ_{q}(s), q \nrest t = p \nrest t$.
\end{enumerate}

We say that a sequence $\langle p_\alpha \such \alpha < \delta \rangle$ is called a \emph{slow fusion sequence} (of length $\delta$) if $\delta \leq \lambda$ and if for $\alpha <\gamma < \delta$, we have $p_\gamma \le_{\gamma,\alpha} p_\alpha$.\end{definition}

Note that this definition depends on the choice of $\langle\kappa_\alpha \such \alpha < \lambda\rangle$.

\begin{lemma}\label{slow_fusion_lemma}
For all the forcings in Definitions \ref{the_forcing} and \ref{more_forcings} and for any  slow fusion sequence $\langle p_\alpha \such \alpha < \delta\rangle$ with $1 \leq \delta \leq \lambda$, we have that $q_\delta=\bigcap\{p_\alpha \such \alpha < \delta\} \in \bP$ and that the following equation from the standard fusion lemma holds:

\vspace{-4mm}

\begin{equation}\label{intermediate}
\begin{split} 
q_\delta = &
\{t \in p_0 \such \exists \alpha<\delta, \exists s \in \Split_{\alpha}(p_{\alpha}), t \sqsubseteq s\} \cup \\
&\bigcup \{\bigcap_{\alpha < \delta} p_\alpha \nrest t \such
t \in  \bigcap\{\Split_{\geq \delta}(p_\alpha) \such \alpha < \delta\})\}.
\end{split}\end{equation}

\noindent where $\Split_{\geq \delta}(p_\alpha)$ has the natural interpretation. If $\delta = \lambda$, then the set in the second line is empty. If $\delta < \lambda$, the right-hand side of the first line is redundant since it is a subset of the set in the second line.

Moreover, $q_\delta \leq_{\delta,\alpha} p_\alpha$ for $\alpha<\delta$, and for $\delta < \lambda$, $\langle p_\alpha \such \alpha < \delta \rangle \concat \langle q_\delta \rangle$ is a slow fusion sequence.\end{lemma}

\begin{proof} We show that the fusion does not break down at intermediate steps or at step $\lambda$ by induction. Specifically, we prove that $q_\delta := \bigcap_{\alpha<\delta}p_\alpha \in \bP$ and that Equation~\eqref{intermediate} holds by induction on $\delta \leq \lambda$. We give the argument for the Miller versions. The argument for the Laver versions is simpler because the antichains $\Split_\alpha(p_\gamma)$ are levels.

We must show that $q_\delta$ as defined is a condition. In the case of $\delta = \eps +1$ we have $q_\delta = p_\eps$ is a condition and the verification of Equation~\eqref{intermediate} is easy. Now suppose that $\delta<\lambda$ is a limit. Since the empty node is 
in any $p_\alpha$, we have $\emptyset \in q_\delta$. The main task is then to show that for each node $s \in q_\delta$, for any $\eps \in [\delta, \lambda)$, there is a $\kappa_\eps$-splitting 
node in $q_\delta$ above $t$.

To find this splitting, first let $\delta$ be the limit of a strictly increasing sequence $\delta_\iota$, $\iota < \cf(\delta)$. By the induction hypothesis and an inspection of Equation~\eqref{intermediate}, there is a strictly increasing sequence $\langle s_\iota \such \iota < \cf(\delta) \rangle$ such that the following hold:

\begin{enumerate}
\item $s_\iota \in \Split_{\delta_\iota}(p_{\delta_\iota})$;
\item for any  $s \in \Split_{\delta_\iota}(p_{\delta_\iota})$ and any $t \in \Succ_{p_{\delta_\iota}}(s)$, $ p_{\delta_\iota} \rest t = p_0 \rest t$;
\item for any $\eps \in [\delta_\iota,\delta)$ we have
\begin{enumerate}[(a)]
\item $\Split_{\delta_\iota}(p_{\delta_\iota} ) = \Split_{\delta_\iota}(p_\eps)$,
\item  for any $s \in \Split_{\delta_\iota}(p_{\delta_\iota})$, $\Succ_{p_{\delta_\iota }}(s) = \Succ_{p_\eps}(s)$.

\end{enumerate}\end{enumerate}

By $(<\!\lambda)$-closure of the conditions (not of $\bP$ itself) we have $r:= \bigcup\{s_\iota \such \iota < \cf(\delta)\} \in q_\delta$. We take $g_\iota $ to be the $\sqsubseteq$-minimal node in $\Split_{\geq \delta}(q_{\delta_\iota})$ that is $\sqsupseteq r$. For any $\iota < \cf(\delta)$, the fourth condition in Definition~\ref{slow_fusion_stuff} of $p_{\delta_\iota} \le_{\delta_\iota,\delta_0} p_{\delta_0}$ tells us that we have  $p_{\delta_\iota} \restriction r = p_{\delta_0} \restriction r$. Hence $g_\iota$ does not depend on $\iota$ and its set of immediate successors also does not depend on $\iota$, so we write $g = g_\iota$ for any $\iota<\cf(\delta)$. Then $g \in \Split_{\geq \delta}(q_\delta)$, $ \Succ_{q_\delta}(g)= \Succ_{p_{\delta_\iota}}(g)$, and for any $t \in \Succ_{q_\delta}(g)$, 
$q_\delta \rest t = p_{\delta_\iota} \rest t$. By $\sqsubseteq$-minimality, we have $g \in \Split_\delta(q_\delta)$. 

By construction we have $\Split_\alpha(p_\alpha) = \Split_\alpha(q_\delta)$ for $\alpha < \delta$, and for $t \in \Split_\alpha(q_\delta)$, $\Succ_{p_{\alpha}}(t) = \Succ_{q_\delta}(t)$. Since for $g \in \Split_\delta(q_\delta)$ $p_0 \rest g = q_\delta \rest g$, the higher part is a perfect tree, as required.
So $q_\delta$ is indeed a condition and Equation~\eqref{intermediate} holds. Then $q_\delta \le_{\delta,\alpha} p_\alpha$ for all $\alpha<\delta$ and the so condition $q_\delta$ can serve as a prolongation of the slow fusion sequence $\langle p_\delta \such \alpha <\delta \rangle$.\end{proof}

The next lemma shows that there are particular prolongations $\seq{p_\alpha}{\alpha<\delta}{}^\frown \langle p_\delta \rangle$ in which $p_\delta$ is stronger than the condition $q_\delta$ from Equation~\eqref{intermediate}.

\begin{lemma}\label{August30}
Let $\langle p_\alpha \such \alpha < \delta \rangle$, $\delta \in \Lim(\kappa)$, be a slow
fusion sequence and let $q_\delta= \bigcap \{p_\alpha \such \alpha < \delta\}$. 
In the Miller case, we can  strengthen $q_\delta$ to a condition $p_\delta$, such that 
for any $t \in \Split_\delta(q_\delta)$, $|\Succ_{p_\delta}(t) | = 1$
and such that 
$\langle p_\delta \such \alpha \leq \delta \rangle$ is a slow fusion sequence.
In the Laver case, we can allow for $t \in \Split_\delta(q_\delta) = \Split_\delta(p_\delta)$ that $\Succ_{p_\delta}(t) $ is a subset of $\Succ_{q_\delta}(t)$.
\end{lemma}

\begin{proof}
For each $t \in \Split_\delta(q_\delta)$ we choose just one 
$f(t) \in \Succ_{q_\delta}(t)$ and some $\beta > \delta$ and some extension $f(t) \sqsubseteq g(t) \in \Split_\beta(q_\delta)$. Then we choose $S(g(t)) \subseteq \Succ_{q_\delta}(g(t))$ that meets the $\kappa_\delta$-splitting criterion and 
let 
\[p_\delta = \bigcup \{q_\delta \rest s \such \exists t \in \Split_\delta(q_\delta), s \in S(g(t))\}.
\]
Then we have for $\alpha < \delta$, $\Split_\alpha(p_\alpha) = \Split_\alpha(p_\delta)$, for any $t \in \Split_\alpha(p_\alpha)$, $\Succ_{p_{\alpha}}(t) = \Succ_{p_\delta}(t)$, 
and for any  $s=g(t) \in \Split_\delta(p_\delta)$
for any $r \in \Succ_{p_\delta}(s) = S(s) $, $p_\delta \rest r = q_\delta \rest r = p_0 \rest r$ and
$\langle p_\alpha \such \alpha \leq \delta\rangle$ is a slow fusion sequence.
In the Laver case we let $\beta = \delta$ and $t = f(t) = g(t)\in \Split_\delta(q_\delta)$ and we choose a subset $\Succ_{p_\delta}(t)$ of $\Succ_{q_\delta}(t)$ that meets the splitting criterion.
\end{proof}

\subsection{Miller Forcing and the Vertical Balcar-Simon Trick}
\label{SubS4.1}

Let $\bP$ be $\Miller_\kappa$, $\Miller_\kappa^\textup{stat}$, $\Miller_\kappa^\textup{club}$. The following does not work either for the Laver case or for the splitting-closed Miller versions.


\begin{definition}[The Vertical Balcar-Simon Trick]\label{vertical}
Let $p \in \bP$. We let $q \prec^\textup{vert} p$ if:

\begin{enumerate}[(i)]
\item $q \leq p$,
\item $\forall t \in \Split_{\alpha+1}(p) \cap q,
\Succ_q(t) = \Succ_p(t)$,
\item for $\alpha =0$ or limit $\alpha < \lambda$, $\forall t \in \Split_\alpha(p),
|\Succ_q(t)|=1$.
\end{enumerate}
\end{definition}

\begin{lemma}\label{existence_0}
For any $p \in \bP_\kappa$ there is some $ q \prec^\textup{vert} p$.
\end{lemma}
\begin{proof}
By induction on $\alpha$ we choose a slow fusion sequence $\langle p_\alpha \such \alpha < \lambda \rangle$ using Lemma~\ref{August30}. We let $p_0 = p$. If $\alpha= \beta+1$ is a successor, we let $p_\alpha = p_\beta$. If $\alpha $ is a limit we first let $q_\alpha = \bigcap \{p_\beta \such \beta < \alpha\}$. By Lemma \ref{slow_fusion_lemma}, $q_\alpha \in \bP_\kappa$. Now for any $t \in \Split_\alpha(q_\alpha)$ we pick some $f(q_\alpha, t) \in \Succ_{q_\alpha}(t)$ and some $g(q_\alpha,t) \sqsupseteq f(q_\alpha,t)$ such that $g(q_\alpha, t) \in \Split_{\alpha+1}(q_\alpha)$.
 We let
\begin{equation*} p_\alpha = \bigcup \{q_\alpha \! \rest \! g(q_\alpha,t) \such  
t \in \Split_\alpha(q_\alpha)\}.
\end{equation*}
Then $\Split_\alpha(p_\alpha) = \Split_{\alpha+1}(q_\alpha)$ and 
$\langle p_\beta \such \beta \leq \alpha \rangle$ is an initial segment of a slow fusion sequence.
By Lemma \ref{slow_fusion_lemma} $q = \bigcap\{p_\alpha \such \alpha < \lambda\}$ is a condition and for any limit $\alpha$ and any $t \in \Split_\alpha(p) \cap q$, $|\Succ_q(t)| = 1$ by construction. Thus we have $ q \prec^{\rm vert} p$.\end{proof}

\begin{theorem}\label{March_31_fruit} $\Miller_\kappa$, $\Miller_\kappa^\textup{club}$,
$\Miller_\kappa^\textup{stat}$
are not $\omega$-distributive. 
\end{theorem}
\begin{proof} Let $\cA_0$ be any maximal antichain in $\bP_\kappa$. Given $\cA_n$, choose for any $p \in \cA_n$, a maximal antichain of conditions $q \prec^\textup{vert} p$ and let $\cA_{n+1}$ be the union of these $q$'s. Then there is no $q \in \bP$ such that for any $n$ there is $p_n \in \cA_n$ with  $q \leq p_n$ because such a $q$ would have no splitting nodes.\end{proof}

\subsection{Laver Forcing and the Horizontal Balcar--Simon Trick}
\label{SubS4.2}

Let $\bP$ be $\Miller_\kappa$, ${}^\textup{scl}\Miller_\kappa$ or $\Laver_\kappa$. Here the $(<\!\cf(\kappa))$-closed versions are of course excluded. The version with stationary ordinal successor sets is also excluded and will be handled in the next subsection.


\begin{theorem}\label{5.16_balcar_simon} If $\bP$ is one of $\Miller_\kappa$, ${}^\textup{scl}\Miller_\kappa$, or $\Laver_\kappa$, then $\bP$ is not countably distributive.
\end{theorem}

Rather than changing splitting nodes to non-splitting nodes, we thin out successor sets. In this way, we prove Theorem \ref{5.16_balcar_simon}, adapting Balcar and Simon's descending technique of their Theorem~\cite[Theorem 5.16]{BalcarSimon88}.

\begin{definition}\label{labeled_tree}
A pair $(p,f_p)$ is called a \emph{labeled tree} if it has the following properties:
\begin{enumerate}
\item $p \in \bP$.
\item $f_p = \langle f_{p,t} \such t\in \Split(p) \rangle$.
\item
For for each $\alpha < \lambda$ for each $s \in \Split_\alpha(p)$ the function
 \[f_{p,s} \colon \OSucc(s) \to \kappa_\alpha
 \]
 is one-to-one and strictly increasing.
\end{enumerate}
We let $L$ be the set of all labeled trees.
\end{definition}

\begin{definition}\label{def:leqstr}
Let $(p,f_p)$ and $(q,f_q)$ be two labeled trees such that $q\leq p$. We write $f_q < f_p$ if 
\[
\forall s \in \Split(q),
\forall  \xi \in \dom(f_{q,s}), f_{q,s}(\xi) <f_{p,s}(\xi).
 \]
 We write $(q,f_q) \prec^\textup{horiz} (p,f_p)$  if $q \leq p$ and
 $f_q < f_p$.
 \end{definition}

\begin{lemma} \label{existence}
Let $(p,f_p)$ be a labeled tree. Then there is some
$(q,f_q)$ such that  $q\leq_0 p$ and $(q,f_q) \prec^\textup{horiz} (p,f_p)$.
\end{lemma}

\begin{proof}
We let $q_{0} = p$, $f_{q_0} = f_p$ and we construct $q$ along with 
$f_{q,s}$, $s \in \Split_\alpha(q_\alpha)$, by a slow fusion sequence
$\langle p_\alpha \such \alpha < \lambda\rangle$. In step $\alpha$ the function $f_{p_\alpha, s}$ will be defined for $s \in \Split_\alpha(p_\alpha)$.

In successor steps, given $p_\alpha$ we define for $s \in \Split_\alpha(p_{\alpha+1})$ a successor set $\Succ_{p_{\alpha+1}}(s)\subseteq \Succ_{p_\alpha}(s)$ and a function $f_{p_{\alpha+1},s}$ as follows: 
Let the range of $f_{p_\alpha,s}$ be enumerated increasingly as $\langle\eta_i \such0 \leq  i < \kappa_{\alpha}\rangle$. We define the lower part of $p_{\alpha+1}$ such that for each $s \in \Split_{\alpha+1}(p_{\alpha}) = \Split_{\alpha+1}(p_{\alpha+1})$ we have
\begin{equation}\label{nonstat}
\Succ_{p_{\alpha+1}}(s) = \{s \concat \xi \such 
\exists i< \kappa_{\alpha}, f_{p,s}(\xi) =  \eta_{i+1}\}.
\end{equation}

The $f_{p,s}$ is not a misprint here: This is the only stage of the construction when we refer to direct successors of $s$. 
Moreover, we define $p_{\alpha+1}$ such that  $\Succ_{p_{\alpha+1}}(s)=\dom(f_{p_{\alpha+1}, s})$ is the set of these $\xi$'s and for $i < \kappa_\alpha$, we let $f_{p_{\alpha+1},s}(\xi) = \eta_i< \eta_{i+1} = f_{p,s}(\xi)$.
The function $f_{p_{\alpha+1},s}$  is also one-to-one and increasing. We define the higher part of the condition $p_{\alpha+1}$ by letting $p_{\alpha+1} \rest t = p_\alpha \rest t$ for $t \in  \Succ_{p_{\alpha+1}}(s)$, $s \in \Split_{\alpha+1}(p_{\alpha+1})$.

In the limit steps $\alpha\leq \lambda$, we first take the intersection $q_\alpha = \bigcap\{p_\beta \such \beta < \alpha\}$ and invoke Lemma~\ref{slow_fusion_lemma}.
For $\alpha<\lambda$, we strengthen $q_\alpha$ to $p_\alpha$ as in the successor step using $\Split_\alpha(q_\alpha)$ as a starting point. According to Lemma~\ref{August30}, $\seq{p_\alpha}{\alpha<\lambda}$ is a slow fusion sequence.

In the end we let $q$ be the fusion of $p_\alpha$, $\alpha < \lambda$, and for $s \in \Split_\alpha(p_\alpha)$ we let $f_{q,s} = f_{p_{\alpha}, s}$.\end{proof}

\begin{definition}
A subset $\cA \subseteq L$ is called an $L$-antichain if for any two $(p,f_p), (q, f_q) \in\cA$, $p$ and $q$ are incompatible.
\end{definition}

\begin{lemma} \label{supplement} Suppose $(p,f_p) \in L$. Suppose that $\cA$ is a maximal $L$-antichain in $L_{(p,f_p)} = \{(q,h) \in L \such (q,h) \prec^\textup{horiz} (p,f_p)\}$.
Then the set $\pi_1(\cA)$ of first components of $\cA$ is a maximal antichain in $\bP$.
\end{lemma}

\begin{proof} Suppose not. Then we may take $q \leq p$ such that $q $ is incompatible with any  member of $\pi_1(\cA)$. By Lemma~\ref{existence} there is $(r,f_r) \prec^\textup{horiz} (q, f_p \rest q)$.
Now $(r, f_r)$ violates the $L$-maximality of $\cA$.
\end{proof}

Now we prove Theorem~\ref{5.16_balcar_simon}.

\begin{proof}
In the Miller case we work with the dense set of conditions $p$ that have a layered sequence $\langle \Split_\alpha(p) \such \alpha < \lambda\rangle$ of splitting maximal antichains and no other splitting.
These are called weakly splitting normal conditions in  \cite[Def. 2.8]{DobrinenHathawayPrikry}.
 These conditions correspond to the $T'$'s and
$\Split_\alpha(T') = \{\bar{\varphi}(t \rest (\kappa_\alpha +1)) \such  t \in \prod_{\beta \leq \alpha+1} \kappa_\beta\}$ of the proof of (1) $\rightarrow$ (2) in 
Lemma~\ref{perfect_equivalence}. Nodes in the maximal antichain $\Split_\alpha(T')$ have $\kappa_\alpha$ immediate successors in the condition.

By induction on $n < \omega$ we choose a sequence maximal $L$-antichains $\seq{\cA_n}{n<\omega}$ such that
\[
\forall m<n<\omega, \forall (r,f_r) \in \cA_n,\exists (q, f_q) \in \cA_m, (r,f_r) \prec^\textup{horiz}(q,f_q).
\]
We let $\cA_0 = \{(p,\id)\}$ where $\id$ is the naturally defined identity function.

We carry out the successor step:
For each $(q,f_q) \in \cA_n$ we take a maximal $L$-antichain $\cA_{n+1}^{(q,f_q)}$ $\prec^\textup{horiz}$-below $(q,f_q)$ and we let 
$\cA_{n+1} = \bigcup \{\cA_{n+1}^{(q,f_q)} \such (q,f_q) \in \cA_n\}$.

\textbf{Claim:} There is no condition $r \in \bP$ such that 
\[
\forall n<\omega, \exists q_n \in \pi_1(\cA_n), r \leq q_n.
\]

Suppose the contrary and fix such an $r$ and a witnessing sequence $\seq{q_n}{n<\omega}$ such that $(q_n,f_{q_n}) \in \cA_n$ and $(q_{n+1},f_{q_{n+1}}) \prec^\textup{hor} (q_n,f_{q_n})$ for all $n<\omega$. 
For $s \in \Split(r)$ and $s \concat \xi \in 
r$ let 
\[
f'_{r,s}(\xi) = \min\{f_{q_n, s}(\xi)  \such n < \omega\}.
\]
This is countable-to-one.
Now focus on $\Stem(r)=s$. It is $\kappa_0$-splitting, and for any $n < \omega$, the function $f_{q_n,s}$ is defined.
For $n < \omega$ let 
\[
X_n = \{\xi \such s \concat \xi \in \Succ_r(s) \wedge f'_{r,s}(\xi) = f_{q_n, s}(\xi)\}.
\]
There is some $n \in \omega$ such that 
$|X_n|= \kappa_0$ (recall that we assume $\cf(\kappa_0)>\omega$). Then for such a $\xi \in X_n$,
\[
f_{q_n, s}(\xi) = f'_{r,s}(\xi) = \min \{f_{q_m, s}(\xi) \such m < \omega\} \leq f_{q_{n+1}, s}(\xi).
\]
However, this contradicts $f_{q_{n+1},s}(\xi) < f_{q_n,s}(\xi)$. This proves the claim.

Thus the maximal antichains $\{\pi_1(\cA_n) \such n < \omega \}$ witness the failure of $\omega$-distributivity.\end{proof}

\subsection{The Stationary Balcar--Simon Trick}
\label{SubS4.3}


Now we present a version of the Balcar-Simon trick that does not apply to the standard cardinal-wise-splitting case, in which splitting into non-stationary sets is permissable. For $\Laver_{\kappa}^\textup{stat}$, where there is splitting into
stationary subsets of $\kappa_{\dom(t)} \cap \cof(\omega)$,
a construction along the lines of Equation~\eqref{nonstat} contradicts Fodor's lemma. We refute $\omega$-distributivity instead via an analysis of Solovay's proof of splitting a stationary subset $S \subseteq \mu \cap \cof(\omega)$
into $\mu$ mutually disjoint stationary sets. The main point is that we use a fixed ladder system.

Let us first recall \cite[Prop 1.9]{Jech_stat}.

\begin{proposition}\label{Jech_elementary_Solovay} Let $\mu$ be an uncountable regular cardinal.
Let $S \subseteq \mu \cap \cof(\omega)$ be a stationary set. Then there are $\mu$ pairwise disjoint stationary subsets $S_i$, $i < \mu$, of $S$.
\end{proposition}
 
\begin{proof} For each $\alpha \in S$, choose an increasing sequence $\langle a_n^\alpha \such n < \omega \rangle$
converging to $\alpha$.

We claim that there is an $n$ such that for all $\eta < \mu$, there are stationarily many $\alpha\in S$ such that 
$a^\alpha_n \geq \eta$.  Otherwise there exists, for each $n$, some $\eta_n$ such that $a^\alpha_n \geq \eta_n$
for only a nonstationary set of $\alpha$'s in $S$. From the $(<\mu)$-completeness of the ideal of non-stationary subsets of $\mu$, it follows that for all but a nonstationary set of $\alpha$'s, the sequences $\langle a^\alpha_n \such n < \omega \rangle$
are bounded by $\sup \{\eta_n \such n < \omega \}$. This is a contradiction.

Thus let $n$ be minimal such that for all $\eta<\mu$, the set $S(\eta) = \{\alpha \in S \such a^\alpha_n \geq \eta\}$ is stationary. For each $\eta$, the function $f(\alpha) = a^\alpha_n$, defined on $S(\eta)$, is regressive and so, by Fodor's Theorem, there is some $\gamma(\eta) \geq \eta$ such that $T_\eta = \{\alpha \in S(\eta) \such  a^\alpha_n = \gamma(\eta)\}$ is stationary. We collect all these $\gamma(\eta)$'s into $\{\eta_{i} \such i < \mu\}$ and let
$S_i = \{\alpha \in S \such a^\alpha_n = \eta_i\}$.\end{proof}

\begin{definition}\label{def:leqstr_2}
Let $\langle\kappa_\beta \such \beta < \lambda \rangle$ be an increasing sequence of regular cardinals converging to $\kappa$ and $\aleph_0 < \lambda = \cf(\kappa) < \kappa_0$. For each $\beta < \lambda$ we
fix a ladder system $\overline{a}_\beta=\langle \langle a^\alpha_{\beta, n} \such n < \omega 
\rangle \such \alpha \in \kappa_\beta \cap \cof(\omega)\rangle$. 
Let $S \subseteq \kappa_\beta$ be stationary. We let
$n_\beta(S)$  (depending on $\bar{a}_\beta$, we indicate only $\beta$) be the least $n$ (from above, with $\kappa_\beta$ now in the role of $\mu$) such that for each $\eta < \kappa_\beta$ 
the set $\{\alpha \in S \such a^\alpha_{\beta,n} \geq \eta\}$ is stationary in $\kappa_\beta$.

For stationary subsets $S, T \subseteq \kappa_\beta \cap \cof(\omega)$, we write $T \subseteq_{\beta,{\rm str}} S$ if $n_\beta(T) > n_\beta (S)$.

Finally for $p, q \in \Laver_{\kappa,{\rm stat}}$ we let $q \prec^\textup{stat} p$ if for $\beta<\lambda$ and any $t \in \Split(q)$ with $\dom(t) = \beta$, $\Succ_q(t) \subseteq_{\beta,{\rm str}} \Succ_p(t)$.
\end{definition}

\begin{lemma} \label{existence_2}
For any $p \in \Laver_{\kappa}^\textup{stat}$  there is $q \prec^\textup{stat} p$.
\end{lemma}

The lemma is proved via a suitable slow fusion sequence, similar to Lemma~\ref{existence}.

\begin{theorem}\label{not_omega_distr_2}
The forcing $\Laver_{\kappa}^\textup{stat}$ is  not $\omega$-distributive. 
\end{theorem}
 
 Now the rest of the proof is as in Theorem~\ref{5.16_balcar_simon}. We use a matrix $\seq{\cA_n}{n < \omega}$ such that for any $p \in \cA_n$ and any $q \in \cA_{n+1}$, $q \leq p$ implies that $q \prec^\textup{stat} p$. If $\seq{p_n}{n<\omega}$ is any descending sequence such that $p_n \in \cA_n$ for all $n < \omega$, then $\seq{p_n}{n<\omega}$ has no lower bound $q$, because for such a $q$ the number $n_{\dom(\Stem(q))}(\OSucc_q(\Stem(q)))$ would not exist.

\section{Failures of Three-Parameter Distributivity}\label{3par-sec}

\subsection{A Cofinal Function from $\omega$ to $\kappa$ in the Pre-Perfect Case}
\label{SubS4.4}

Let $\bP$ be any of the pre-perfect versions and let $\lambda$ be uncountable. In the Miller case, we restrict the forcing to the dense set of conditions $p$ in which  only the nodes in 
$\bigcup \{\Split_\alpha(p)\such \alpha < \lambda\}$ have more than one successor.

We use the technical notion of a minimal branch.

\begin{definition}\label{minimal_branch}
 Let $\chi > (2^\kappa)^+$ be a regular cardinal and let $<_\chi$ be a well-ordering of $H(\chi)$, the set of sets of hereditary cardinality $<\chi$.

For any $t \in p \in {}^{\textup{pre}}\bP_\kappa$ we let $b_{\rm min}(p,t)$ be the $<_\chi$-least branch of $p$ such that $t \in b$.
\end{definition}

\begin{definition}\label{pre-perfect_splitting}
For each $p \in {}^{\textup{pre}}\bP_\kappa$  we define its standardized pre-perfect subtree $p^{\textup{pre}}$ 
as follows. By induction on $\alpha< \lambda = \cf(\kappa)$ we define subsets $S(p,\alpha)\subseteq p$ and $\Split^{\textup{pre}}(p,\alpha)\subseteq \Split_\alpha(p)$:

\begin{enumerate}

\item $\Split^{\textup{pre}}(p,0) = \{\Stem(p)\}$.

\item For $\beta < \lambda$, we let $S(p,\beta) = \bigcup \{\Succ_p(t) \such t \in \Split^{\textup{pre}}(p,\beta)\}$.

\item
We let 
\[\Split^{\textup{pre}}(p,\alpha+1) = \{ t \in \Split_{\alpha+1}(p) \such
\exists s \sqsubseteq t, s \in S(p,\alpha)\}.
\]
\item For $\delta < \lambda$ limit we let
\begin{align*}
\Split^{\textup{pre}}(p,\delta) =  
\{t \in \Split_\delta(p) \such  & \exists \alpha < \delta, \exists s \in S(p,\alpha), \exists \eps < \lambda,
\\
&
 t = (b_{\rm min}(p,s) \nrest \eps)  \in \Split_\delta(p)\}.
\end{align*}


\end{enumerate}

We let 
\[
p^{\textup{pre}} = \{s \in p \such \exists \alpha < \lambda, \exists t \in S(p,\alpha),  s \sqsubseteq  t\} .
\]
\end{definition}

Then $p^{\textup{pre}} \in {}^{\textup{pre}}\bP_\kappa$ and $p^{\textup{pre}} \subseteq p$. Note that $(p^{\textup{pre}})^{\textup{pre}} = p^{\textup{pre}}$ and  no splitting antichains disappear, i.e$.$,
\[
\emptyset \neq \Split^{\textup{pre}}(p^{\textup{pre}},\alpha)= \Split_\alpha(p^\textup{pre})  \subseteq \Split_\alpha(p).
\]
In the Laver case,
 $p^{\textup{pre}}( \dom(\Stem(p^{\textup{pre}}))  + \alpha) = \Split^{\textup{pre}}(p^{\textup{pre}},\alpha)$.
In addition, for any $q \leq p^{\textup{pre}}$ we have $q = q^{\textup{pre}}$.

\begin{definition}
We call $p$ \emph{standardized} if $p = p^{\textup{pre}}$.
\end{definition}


The construction in Definition~\ref{pre-perfect_splitting} implies the following:

\begin{fact} The set of standardized pre-perfect trees  is an open dense subset of $\bP^{\textup{pre}}$.\end{fact}

For $s \sqsubseteq t \in p$ we write $[s,t) = \{ r \in p \such s \sqsubseteq r \sqsubset t\}$. We let $t_{-1} = \emptyset$.

We use the following, which in the special case of allowing only eventually zero nodes is proved in \cite[Lemma 7.1]{DobrinenHathawayPrikry}:

\begin{proposition} For $p= p^{\textup{pre}} \in \bP$,  any node $s \in p$, has only finitely many  different $b_{\rm min}(p,t)$ for $t \in s$.\end{proposition}

\begin{proof} This is shown by reading  
 the recursive definition of $p=p^{\textup{pre}}$. By induction on $\alpha$, we see that $\forall t \in p(\alpha)$, $\{b_\text{min}(s,p) \such s \sqsubseteq t \}$ is finite.\end{proof}

\begin{definition} Fix $p = p^{\textup{pre}}$. We let $t_{i-1} = \emptyset$.
\begin{align*}
\dot{f_p} = \bigl\{ \langle \check{(i, \alpha_i)}, q \rangle \such & i< n,
q \leq p, \;  \exists t_0 \dots \exists t_{n-1} 
\\
& \bigwedge_{0 \leq i <n} \bigl(t_i \sqsubseteq \Stem(q)  \wedge \dom(t_i) = \alpha_i \wedge \\
& \forall t \in [t_{i-1}, t_i), 
b_{\rm min}(p, t_{i-1})  = b_{\rm min}(p,t) \neq b_{\rm min}(p,t_i)
\bigr)
\bigr\}.
\end{align*}
\end {definition}

\begin{lemma}
The condition $p$ forces that $ \dot{f_p}$ is a cofinal function from $\omega$ to $\lambda$.
\end{lemma}

\begin{proof}
Let $\alpha  \in \lambda$ and let $q \leq p \in \bP$. There is some $n \in \omega$ such that $q \Vdash \dot{f_p} \rest n = r$  for some $r \in V \cap {}^n \lambda$ and $q$ determines just $\dot{f} \rest n$, that means that there are $t_i$, $0 \leq i < n$, $t_i \sqsubseteq \Stem(q)$, as in the definition of $\dot{f}$ and
$\dom(t_{i}) = r(i)$ and $b_{\rm min}(p,t_{i-1}) = b_{\rm min}(p, s)$ for $t_{i-1} \sqsubseteq s \sqsubset t_i$ and $b_{\rm min}(p,t_{n-1}) = b_{\rm min}(p, s)$ for $t_{n-1} \sqsubseteq s \sqsubseteq \Stem(q)$. In other words, on $\{s \such s \sqsubseteq \Stem(q)\}$ we find exactly $n$ different values of $b_{\rm min}(p,s)$, represented by
$\{b_{\rm min}(q,t_i) \such 0 \leq i < n\}$. We assume that $r(n-1) < \alpha$.

 We choose $t \in b_{\rm min}(p, t_{n-1}) \cap \Split(q)$ such that $t_{n-1} \sqsubseteq t$ and $\dom(t) \geq \alpha$. Then we take $t_n \in \Succ_q(t)$ with
$b_{\rm min}(p,t_n) \neq b_{\rm min}(p,t)$. We let $q' = q \restriction t_n$.
Then $q' \Vdash \dot{f}(n) = \dom(t_n)$ and $\dom(t_n) \geq \alpha +1$.
\end{proof}

\begin{remark} By definition of $\dot{f_p}$, $\rge(\dot{f_p})$ contains only successor ordinals. In all trees  that appear as conditions in any of our forcings any increasing sequence of nodes has at most one limit node, namely their union. Sometimes, these trees are called set-theoretic trees.\end{remark}

Now we take a maximal antichain $A$ of $p \in \bP$ such that $p = p^{\textup{pre}}$ and
let $\dot{f}$ be the mixing of $\{\dot{f_p} \such p \in A\}$.
Then we have:

\begin{theorem}\label{cofinal_fct}
$\bP$ forces that $\dot{f}$ maps $\omega$ cofinally into $\lambda$.
\end{theorem}

This means for that $\bP$ collapses $\lambda$ to $\omega$ if $\lambda = \omega_1$.

\begin{corollary}\label{Maxwell_by_email}
The forcing $\bP$ is not $(\lambda,2)$-distributive.
\end{corollary}
\begin{proof} Let $f$ be a cofinal function from Theorem~\ref{cofinal_fct} and $B = \rge(f)$. The characteristic function of $B$ is a new function from $\lambda$ to $2$.
\end{proof}

\subsection{The Singular Perfect Case}

Now we further develop our non-distributivity results from Section~\ref{balcar-simon-sec}.

\begin{theorem}\label{sing-three-par} Let $\kappa$ be a singular cardinal of cofinality $\lambda>\omega$ and let $\bP$ be any of $\M_\kappa^\textup{}$, $\Laver_\kappa^\textup{}$, $\M_\kappa^\textup{stat}$, or $\Laver_\kappa^\textup{stat}$. Then $\bP$ is not $(\omega,\cdot,\lambda^+)$-distributive.\end{theorem}

\begin{lemma}\label{main-lemma} Let $p \in \bP$ be a condition and suppose that $\mathcal{A} = \{q_\xi:\xi<\lambda\}$ is a maximal antichain below $p$. Then for all $t \in p$, there is some $s \in p \rest t$ and some $\xi<\lambda$ such that $p \rest s \le q_\xi$.\end{lemma}

\begin{proof} Fix such a $p \in \bP$ and $\mathcal{A}$.

\textbf{Claim:} Fix $\zeta<\lambda$. For all $\bar t$, there is some $t \in p \rest \bar t$ such that one of the following holds:

\begin{enumerate}[(i)]
\item $\forall \xi < \zeta$, $(p \rest t) \cap q_\xi \subseteq t$,
\item $\exists \xi<\zeta$, $p \rest t \le q_\xi$.
\end{enumerate}

Suppose that the claim is false. The negation of this statement is that there is some $\bar t \in p$ such that $\forall t \in p \rest \bar t$, we have both of the following:

\begin{enumerate}[(i')]
\item $\exists \xi < \zeta$, $s \in p \rest t$, $s>t$ and $s \in (p \rest t) \cap q_\xi$,
\item $\forall \xi < \zeta$, $\exists s \in p \rest t$, $s>t$ and $s \in (p \rest t) \setminus q_\xi$.
\end{enumerate}

This implies that for all $\xi < \zeta$, $t \in (p \rest \bar t) \cap q_\xi$, there is some $\eta<\zeta$, $\eta \ne \xi$ and some $ s \in (p \rest t) \cap q_\eta$; this is because we can first take $s \in (p \rest t) \setminus q_\xi$ using (ii') and then $s' \in (p \rest s) \cap q_\eta$ for some $\eta<\zeta$ using (i'), from which it will follow that $\eta \ne \xi$. It then follows that we can find a $\sqsubseteq$-increasing sequence of nodes $\seq{ s_i}{i<\zeta} \subset p$ such that for all $i<\zeta$, $s \in q_{\xi(i)}$ for $\xi(i)<\zeta$ and yet $s_i \notin q_{\xi(j)}$ for $j<i$. If $s = \bigcup_{i<\zeta} s_i$, then $s \in p$. Then for all $t \sqsupseteq s$ with $ t \in p \rest s$ and for all $\xi<\zeta$, we have $t \notin q_\xi$. Hence the claim holds as witnessed by $s$ via \emph{(i)}.

Now suppose that the statement of the lemma is false: $\exists \bar t \in p$, $\forall t \in p \rest \bar t$, $\xi<\kappa$, $p \rest t \not\le q_\xi$.

Suppose that $\bP$ is one of the Miller versions. It follows from the claim that for all $t \in p \rest \bar t$, $\zeta < \lambda$, there is some $s \in p \rest t$ such that $\forall \xi < \zeta$, $(p \rest s) \cap q_\xi \supseteq t$. We can therefore build a sequence $\seq{A_\zeta}{\zeta<\lambda}$ of maximal antichains in $p \rest \bar t$ such that the following hold:

\begin{enumerate}
\item $t \in A_\zeta$, $t$ is a splitting node,
\item for all $\xi<\zeta$, $(p \rest t) \cap q_\xi \subseteq t$,
\item for all $\zeta < \zeta'<\lambda$, $t \in A_\zeta$, $t' \in A_{\zeta'}$, $t \sqsubseteq t'$ or $t \perp t'$.
\end{enumerate}

Then let $r$ be the downwards closure of $\bigcup_{\zeta<\kappa}A_\zeta$. Then $r$ is a condition in $\bP$ both in the Miller and Laver cases. Moreover, for all $\xi<\lambda$, $r \perp q_\xi$, and so $\mathcal A$ is not maximal after all.\end{proof}

Here is the proof of Theorem~\ref{sing-three-par}.

\begin{proof} Let $\{\mathcal{A}_n \such n \in \omega \}$ be a matrix that witnesses the failure of $\omega$-distributivity. Fix $p \in \bP$ and suppose for contradiction that for all $n<\omega$, $|\{r \in \mathcal{A}_n : r \| p \}| \le \lambda$.

Given $r \in \mathcal{A}_n$ such that $r \| p$, let

\vspace{-3mm}

$$\pi(r) = \{t \in p \cap r : (p \cap r) \rest t \text{ contains a perfect subtree}\}.$$

\noindent Let $\mathcal{A}_n^* = \{\pi(r) : r \in \mathcal{A}_n, r \| p\}$. Then $\mathcal{A}_n^\ast$ is a maximal antichain of size $\le \lambda$ below $p$, and the downwards closures of the $\mathcal{A}_n^\ast$'s have empty intersection.

Now define an increasing sequence $\seq{t_n}{n<\omega}$ of nodes through $p$ as follows: Let $t_0$ be the root, and if $t_n$ is defined apply Lemma~\ref{main-lemma} to $t_n$ and $\mathcal{A}_n^\ast$. Let $t^\ast = \bigcup_{n<\omega}t_n$. Then for all $n<\omega$, there is a unique $q_n \in \mathcal{A}_n^\ast$ such that $p \rest t^\ast \le q_n$. But this is impossible, contradiction.\end{proof}

\subsection{Collapses for Naive Higher Sacks Forcings}

In this section we use failures of three parameter distributivity to obtain collapses for perfect tree forcings for regular cardinals.


We define a basic tree forcing that consists merely of perfect binary trees in a regular cardinal $\kappa$. This will suffice to demonstrate an isomorphism to the L{\'e}vy collapse, but the reader can observe that the result can be generalized to other forms of splitting behavior.

\begin{definition} Let $\kappa$ be a regular cardinal. Then $\bT$ will consist of conditions $p \subset 2^{<\kappa}$ such that:

\begin{enumerate}
\item $\forall t \in p$, $t \in 2^{<\alpha}$, $\alpha<\kappa$,
\item $\forall t \in p$ and $s \sqsubseteq t$, then $s \in p$,
\item $\forall t \in p$, $\exists s \sqsupseteq t$, $s {}^\frown 0, s {}^\frown 1 \in p$,
\item if $\alpha<\kappa$ is a limit, $s \in 2^\alpha$, and for all $\beta<\alpha$, $s \rest \beta \in p$, then $s \in p$.

\end{enumerate}\end{definition}

So there is no restriction on splitting as there is in Kanamori's version. The main theorem of this subsection is the following:


\begin{theorem}\label{naive-sacks} If $2^\kappa = \kappa^+$, then $\bT \simeq \Levy(\omega,\kappa^+)$.\end{theorem}


We adapt the methods from the previous subsection and combine them with the theory amalgamated in the following:

\begin{fact}[From the Prague School]\label{relevant-theory} Suppose that $\bQ$ is a poset of cardinality $\nu$ where $\nu$ is regular and that $\bQ$ is not $(\omega,\cdot,\nu)$-distributive. Then $\Vdash_\bQ ``|\nu|=\omega$''. Moreover, $\bQ$ is forcing-equivalent to $\Levy(\omega,\nu)$.\end{fact}


\begin{proof}[Sketch] We include an outline of the concepts at play. It is well-known that any forcing of cardinality $\nu$ that collapses $\nu$ to have countable cofinality is forcing-equivalent to $\Levy(\omega,\nu)$ \cite{Jech3}. This gives us the last part of the statement.

Now we want to justify the assertion that $\nu$ will be collapsed under these circumstances. First of all, it is enough to show that $\Vdash_\bQ ``\cf(\nu)=\omega$'' because an argument of Sakai tells us that any $\nu$-sized poset the singularizes $\nu$ to have cofinality $\omega$ will actually collapse $\nu$ to be countable \cite{Sakai_semiproper}. Now suppose $\mathcal{A}_n,n<\omega$, witness the failure of $(\omega,\cdot,\nu)$-distributivity. Then one can observe that if we have enumerations $\mathcal{A}_n=\langle q_\xi^n : \xi < \nu \rangle$ for all $n<\omega$, then
\[
 \dot c := \{  \langle \check{(n,\xi)},q^n_\xi \rangle \such n < \omega, \xi < \nu \}
 \]
\noindent is a name for a function singularizing $\nu$ to have cofinality $\omega$.\end{proof}


\begin{proposition} $|\bT| = 2^\kappa$.\end{proposition}

\begin{lemma} The $2^\kappa$-chain condition fails everywhere in $\bT$.\end{lemma}

\begin{proof} Let $E$ be the set of ordinals $\alpha<\kappa$ such that either $\alpha$ is a limit or $\alpha = \beta + n$ where $\beta$ is a limit and $n<\omega$ is even. Consider the set functions $f: E \to 2$, and let $T_f$ for such an $f$ consist of all $t: \gamma \to 2$ where $\gamma < \kappa$ and such that $t \rest E = f \rest \gamma$. If $f, g : E \to 2$, $f \ne g$, and $\alpha \in E$ is such that $f(\alpha) = g(\alpha)$, then $T_f$ and $T_g$ do not share any nodes of length $>\alpha$, so $T_f$ and $T_g$ are compatible. This construction can also be done below any condition.\end{proof}

Now we prove Theorem~\ref{naive-sacks}.

\begin{proof} By Fact~\ref{relevant-theory}, it is enough to show that $\bT$ fails to be $(\omega,\cdot,\kappa^+)$-distributive. 

We employ the vertical Balcar--Simon trick. For every $p \in \bT$, there is some $q \in \bT$ with $q \prec^\textup{vert} p$. In this way define a sequence of $2^\kappa = \kappa^+$-sized maximal antichains $\seq{\mathcal{A}_n}{n<\omega}$
such that for any $n$ for any $p \in \cA_n$ for any $q \leq p$, $q \in \cA_{n+1}$ we have $q \prec^\textup{vert} p$.
 If $\seq{p_n}{n<\omega}$ is $\prec^\textup{vert}$-decreasing and $q \le p_n$ for all $n<\omega$, then $q$ has no splitting.
   The fact that the sequence $\seq{\mathcal{A}_n}{n<\omega}$ witnesses the failure of $(\omega,\cdot,\kappa^+)$-distributivity is exactly analogous to the argument for Lemma~\ref{main-lemma}.\end{proof}

\section{Further Directions}
\label{S6}

Many issues remain in the uncountable case in terms of the non-closed versions, and some of these issues may be independent of $\textsf{\textup{ZFC}}$.

\begin{problem} What are the complete sub-forcings 
of those forcings we have studied here?
\end{problem}

It is possible that given the different behavior in terms of regularity between the Miller and Laver cases (see \ref{extensions}), additional assumptions about diamond principles or ladder systems may be necessary to settle these questions.

\begin{problem} To what extent do these forcings collapse cardinals?\end{problem}

Note that certain collapses of successors of singulars involve $\textup{\textsf{PCF}}$ theory and large cardinals \cite{Cummings-collapses}.

\end{document}